\documentclass[10pt]{amsart}
\usepackage{latexsym, amsmath,amssymb, hyperref, import}

\usepackage{color}

\usepackage{amsthm}

\newtheorem{theorem}{Theorem}

\numberwithin{equation}{section}
\numberwithin{theorem}{section}
\newtheorem{lemma}[theorem]{Lemma}
\newtheorem{corr}[theorem]{Corollary}
\newtheorem{conjecture}[theorem]{Conjecture}

\newtheorem{proposition}[theorem]{Proposition}
\newtheorem{deff}[theorem]{Definition}
\newtheorem{remark}[theorem]{Remark}
\newcommand{\bth}{\begin{theorem}}
	\newcommand{\ble}{\begin{lemma}}
		\newcommand{\bcor}{\begin{corr}}

			\newcommand{\bdeff}{\begin{deff}}

				\newcommand{\bprop}{\begin{proposition}}
					\newcommand{\ele}{\end{lemma}}
				\newcommand{\ecor}{\end{corr}}
			\newcommand{\edeff}{\end{deff}}
		
		\newcommand{\eprop}{\end{proposition}}

	\newcommand{\la}{\lambda}
\newcommand{\R}{\mathbb R}
\newcommand{\C}{\mathbb C}
\newcommand{\Z}{\mathbb Z}

\newcommand{\bfk}{\mathbf k}
\newcommand{\lap}[1]{\sqrt{-\Delta_{#1}}}
\newcommand{\I}{\mathcal I}
\newcommand{\tgamma}{\tilde \gamma}

	\newcommand{\eps}{\varepsilon}
	
		\newcommand{\inj}{\text{inj}}

	\newcommand{\supp}{\text{supp }}
	\renewcommand{\Pi}{\varPi}

	\newcommand{\Rt}{{\Bbb R}^2}

	\newcommand{\tg}{{\tilde g}}
	\newcommand{\dgt}{d_{\tilde g}}
	\newcommand{\sqrtg}{\sqrt{-\Delta_g}}
	\newcommand{\sqrtd}{\sqrt{-\Delta_{\tilde g}}}
	\newcommand{\Pe}{\sqrt{-\Delta}}

	\thanks{ }
	
	\usepackage{graphicx}

\begin{document}

		\title[Improved Generalized Periods Over Curves]
		{Improved Generalized Periods Estimates over Curves on Riemannian Surfaces with Nonpositive Curvature}

		\begin{abstract}  
				We show that on compact Riemann surfaces of nonpositive curvature, the generalized periods, i.e. the $\nu$-th order Fourier coefficients of eigenfunctions $e_\lambda$ over a closed smooth curve $\gamma$ which satisfies a natural curvature condition,  go to 0 at the rate of $O((\log\lambda)^{-1/2})$, if $0<|\nu|/\lambda<1-\delta$, for any fixed $0<\delta<1$.  Our result implies, for instance, the generalized periods over geodesic circles on any surfaces with nonpositive curvature would converge to zero at the rate of $O((\log\lambda)^{-1/2})$. A direct corollary of our results and the QER theorem of Toth and Zelditch \cite{TZ} is that for a geodesic circle $\gamma$ on a compact hyperbolic surface, the restriction $e_{\lambda_j}|_\gamma$ of an orthonormal basis $\{e_{\lambda_j}\}$ has a full density subsequence that goes to zero in weak-$L^2(\gamma)$. One key step of our proof is a microlocal decomposition of the measure over $\gamma$ into tangential and transversal parts. 
		\end{abstract}
					\author{Emmett L. Wyman}
		\address{Department of Mathematics, Johns Hopkins University, Baltimore, MD 21210}
		\email{ewyman3@math.jhu.edu}
		\author{Yakun Xi}
		\address{Department of Mathematics, University of Rochester, Rochester, NY 14627}
		\email{yxi4@math.rochester.edu}

		\maketitle
		
	%%%%%%%%%%%%%%%%
	% INTRODUCTION %
	%%%%%%%%%%%%%%%%
		
	\section{Introduction}

Let $e_\la$ denote the $L^2$-normalized eigenfunction on a compact, boundary-less Riemannian surface $(M,g)$, i.e., 
$$-\Delta_g e_\la=\la^2 e_\la, \quad \text{and } \, \, \int_M |e_\la|^2 \, dV_g=1.$$
Here $\Delta_g$ is the Laplace-Beltrami operator on $(M,g)$ and $dV_g$  is the volume element associated with metric $g$.

Various questions concerning analytic properties
of eigenfunctions have drawn much attention from number theorists, analysts and physicists in recent years. In particular, it is an area of interest to study  quantitative behaviors of eigenfunctions restricted to smooth curves. Recently, there are a lot of interests in studying the period integral of eigenfunctions over smooth closed curves on compact hyperbolic surfaces due to its significance in number theory. See e.g. \cite{Zel}, \cite{Pitt}, \cite{Rez} and the references therein.

Using the Kuznecov formula, Good \cite{Good} and Hejhal \cite{Hej}  proved independently that if $\gamma_{per}$ is a periodic geodesic on a compact hyperbolic
surface $M$ parametrized by arc-length, then, uniformly in $\la$,
\begin{equation}\label{i.1}
\Bigl|\, \int_{\gamma_{per}} e_\la \, ds\, \Bigr| \le C_{\gamma_{per}}.
\end{equation}
A few years later, Zelditch~\cite{ZelK} generalized this result by showing that if $\la_j$ are the eigenvalues of $\sqrt{-\Delta_g}$ for an orthonormal basis of eigenfunctions $e_{\la_j}$ on
a compact Riemannian surface, and if $p_j(\gamma_{per})$ denote the period integrals  of $e_{\la_j}$ as in \eqref{i.1},
then
\begin{equation}\label{Zel}\sum_{\la_j\le \la}|p_j(\gamma_{per})|^2 =c_{\gamma_{per}} \la +O(1),\end{equation}
where the remainder being $O(1)$ implies \eqref{i.1}.  Further work for hyperbolic surfaces giving more information about the lower order remainder in terms
of geometric data of $\gamma_{per}$ was done by Pitt~\cite{Pitt}.  By Weyl's Law, the number of eigenvalues (counting multiplicities) that are smaller than
$\la$ is about $\la^2$, and thus \eqref{Zel} implies that, on average, one can do much better than \eqref{i.1}. It was pointed out by Chen and Sogge \cite{CSPer} that \eqref{i.1} is sharp on both the sphere $S^2$ and the flat torus $\mathbb T^2$.  On $S^2$,  \eqref{i.1} is saturated by  $L^2$-normalized zonal spherical harmonics of even degree restricted to the equator, while it is trivially sharp on the flat torus ${\mathbb T}^2$.  In contrast, as an analogy with the Lindel\"of conjecture for certain $L$-functions, it is conjectured by Reznikov \cite{Rez} that the period integrals over closed geodesics/ geodesic circles on a compact hyperbolic surface satisfy the following:
\begin{conjecture}[\cite{Rez}]Let $\gamma$ be a periodic geodesic or a geodesic circle on a compact hyperbolic surface $(M,g)$. Then given $\eps>0$, there exists a constant $C_\eps$ depending on $\eps$, $M$ and the length of $\gamma_{per}$, such that 
	\begin{equation}	\Bigl|\, \int_{\gamma}\, e_\la \, ds\, \Bigr| \le C_\eps\la^{-\frac12+\eps}.
	\end{equation}		
	
\end{conjecture}

In a paper of Chen and Sogge \cite{CSPer}, the first improvement over \eqref{i.1} for closed geodesics was obtained. Indeed, they proved a stronger statement saying that the period integrals in \eqref{i.1} converge to 0 as $\la\to \infty$, if $(M,g)$ has strictly negative
curvature.  The proof exploited the simple geometric fact that, due to the presence of negative curvature, there is no non-trivial geodesic rectangle on the universal cover of $M$.  This allowed them to show that the period integrals goes to 0 as $\la\rightarrow\infty$. In a recent paper of Sogge, the second author and Zhang \cite{Gauss}, this method was further refined, and they managed to show that 
\begin{equation}\label{gauss}
\Bigl|\, \int_{\gamma_{per}} e_\la \, ds\, \Bigr| = O((\log\lambda)^{-\frac12}),
\end{equation}
under a weaker curvature assumption, where the curvature $K=K_g$ of $(M,g)$ is  assumed to be non-positive but allowed to vanish at an averaged
rate of finite type. The key idea of \cite{Gauss} was to use the Gauss-Bonnet Theorem to get a quantitative version of the ideas used in \cite{CSPer}, that is, to quantitatively avoid geodesic rectangles on the universal cover.  

Since then, there have been plenty new developments in this area.  The first author (\cite{emmett2}, \cite{emmett1}) generalized the results in \cite{CSPer} and \cite{Gauss} to curves which have geodesic curvature bounded away from the curvature of limiting circles. See also the recent work of Canzani, Galkowski \cite{canzani}
for $o(1)$ bounds for averages over hypersurfaces under weaker 
assumptions. To describe this result, we shall need to introduce a few notations. If $\gamma$ is a smooth curve in $M$, we denote by $\kappa_\gamma(t)$ the geodesic curvature of $\gamma$ at $t$, i.e.
\[
\kappa_\gamma(t) = \frac{1}{|\gamma'(t)|} \left| \frac{D}{dt} \frac{\gamma'(t)}{|\gamma'(t)|} \right|,
\]
where $D/dt$ is the covariant derivative in the parameter $t$. Fixing a point $p \in M$ and $v \in T_pM$, we denote by $v^\perp$ a choice of vector in $T_pM$ for which $|v^\perp| = |v|$ and $\langle v^\perp, v \rangle = 0$. We also need a function $\mathbf k$ on the unit sphere bundle of $M$ representing the ``curvature of a limiting circle."
\begin{deff}[Curvature of a limiting circle] \label{def k}
	For a point $p \in M$ and $v \in S_pM$ and let $r \mapsto \zeta(r)$ the unit speed geodesic with $\zeta(0) = p$ and $\zeta'(0) = v$. Let $J$ be a Jacobi field along $\zeta$ satisfying
	\begin{equation} \label{J initial condition}
	J(0) = \zeta'(0)^\perp.
	\end{equation}
	We let $\bfk_p(v)$ denote the unique number such that
	\begin{equation} \label{J bounded}
	|J(r)| = O(1) \quad \text{ for } r \in (-\infty,0]
	\end{equation}
	if $J$ satisfies the additional initial condition
	\begin{equation} \label{J' initial condition}
	\frac{D}{dr} J(0) = \bfk_p(v) J(0).
	\end{equation}
\end{deff}

The name ``curvature of a limiting circle" will be clear after a lift to the universal cover. By the Cartan-Hadamard Theorem, we identify the universal cover of $M$ with $(\R^2, \tilde g)$, where $\tilde g$ is the pullback of the metric tensor $g$ through the covering map. If $p \in M$ and $v \in T_pM$, denote by $\tilde p$ and $\tilde v$ their respective lifts to $\tilde M$ and $T_{\tilde p}\tilde M$. Then $\mathbf k_p(v)$ denotes the limiting curvature of the circle at $\tilde p$ with center taken to infinity along the geodesic ray in direction $-\tilde v$. In the flat case, $\mathbf k_p(v) = 0$ for all $p\in M$, while if $M$ is a compact hyperbolic surface with sectional curvature $-1$, then $\bfk_p(v) = 1$ for all $p\in M$, which equals the curvature of a horocycle in the hyperbolic plane. See Figure 1.
\begin{figure}
	\centering
	\includegraphics[width=.65\textwidth]{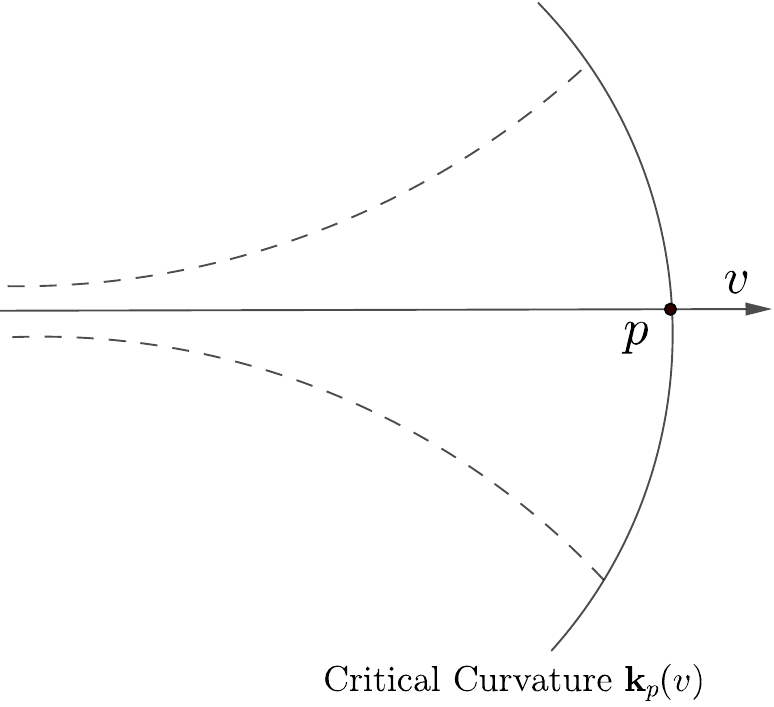}
	\caption{}
	\label{fig1}
\end{figure}		
\begin{theorem}[\cite{emmett2}] \label{W}
	Let $(M,g)$ have nonpositive sectional curvature and let $\gamma$ be a smooth closed unit-speed curve in $M$. Then we have
	\[
	\left|\int_\gamma \,e_\lambda \, ds\right| = O((\log \lambda)^{-1/2}),
	\]
	provided that
	\begin{equation} \label{curvature hypotheses}
	\kappa_\gamma(s) \neq \bfk_{\gamma(s)}(\gamma'^\perp(s)) \quad \text{ and } \quad \kappa_\gamma(s) \neq \bfk_{\gamma(s)}(-\gamma'^\perp(s))
	\end{equation}
	for all points $\gamma(s)\in\gamma$.
\end{theorem}

A natural way to look at the period integral \eqref{i.1}  is to regard it as the 0-th order Fourier coefficient of $e_\la|_{\gamma_{per}}$. The general $\nu$-th order Fourier coefficients of $e_\la|_{\gamma_{per}}$ are called {\it generalized periods} (see \cite{Rez}). Generalized periods for hyperbolic surfaces naturally arise in the theory
of automorphic functions, and are of interest in their own right, thus have been studied considerably by number theorists.  In the paper \cite{Rez}, Reznikov showed that on compact hyperbolic surfaces, if $\gamma$ is a periodic geodesic or a geodesic circle, the $\nu$-th order Fourier coefficients of $e_\lambda|_\gamma$ is uniformly bounded if $\nu\le c_\gamma\la$ for some constant $c_\gamma$ depending on $\gamma$.  In this spirit, the second author \cite{inner} generalized Reznikov's results to arbitrary smooth closed curves over arbitrary compact Riemannian surfaces.
\begin{theorem}[\cite{inner}]\label{inner}
	Let $\gamma$ be a smooth, closed, unit-speed curve on $(M,g)$. Let $|\gamma|$ denote its length. Given $0<c<1$, if $\nu$ is an integer multiple of $2\pi|\gamma|^{-1}$ such that $0\le\frac{|\nu|}{\la}\leq c<1$,  then we have
	\begin{equation}\label{period}
	\Big|\int_{\gamma}e_\lambda(\gamma(s)) e^{-i\nu s}\, ds\Big|\le C |\gamma|\|e_\la\|_{L^1(M)},
	\end{equation}
	where the constant $C$ only depends on $(M,g)$ and $c$, and will be uniform if $\gamma$ belongs to a class of curves with bounded geodesic curvature.
\end{theorem}
As for the period integrals, the above bounds are sharp for surfaces with constant non-negative curvature.  (See \cite[Section 5]{inner}.) \eqref{period} improves \eqref{i.1} for the case $\nu=0$ by having an $L^1(M)$ norm on the right hand side, and it trivially implies these generalized periods are bounded,
	\begin{equation}\label{period'}
\Big|\int_{\gamma}e_\lambda(\gamma(s)) e^{-i\nu s}\, ds\Big|\le C |\gamma|.
\end{equation} We also remark that the frequency gap condition $0\le\frac{\nu}{\la} \leq c < 1$ is necessary, in the sense that, at the resonant frequency $\nu=\la$, \eqref{period} fails to hold on $S^2$. Indeed, on $S^2$,  $L^2$-normalized highest weight spherical harmonics with frequency $\la$ restricted to the equator have $\nu$-Fourier coefficient $\sim\nu^\frac14$, which represents a big jump from \eqref{period}.

Another key insight provided by 
\cite{Gauss} and \cite{emmett2} is that under suitable curvature assumptions (for both $\gamma$ and $M$), the $\nu$-th Fourier coefficients of $e_\lambda|_{
	\gamma}$ satisfies
\begin{equation}\label{gauss'}
\Bigl|\, \int_{\gamma} e_\la e^{-i\nu s} \, ds\, \Bigr| \le C_\nu|\gamma|(\log\lambda)^{-\frac12},
\end{equation} where $C_\nu$ is a constant times a positive power of $\nu$. On the other hand, it is  conjectured in \cite{Rez} that for compact hyperbolic surface, we should expect much better estimates.
\begin{conjecture}[\cite{Rez}]\label{C2}Let $\gamma$ be a closed geodesic or a geodesic circle on a compact hyperbolic surface $(M,g)$. Then given $\eps>0$, $0<c<1$, there exists a constant $C_\eps$ depending on $\eps$, $M$ and the length of $\gamma_{per}$, such that for $0\le\frac{\nu}{\la}\leq c <1$, we have
	\begin{equation}	\Bigl|\, \int_{\gamma_{per}} e_\la(\gamma_{per}(s)) e^{-i\nu s} \, ds\, \Bigr| \le C_\eps\la^{-\frac12+\eps}.
	\end{equation}		
	
\end{conjecture}
The above conjecture, if true, illustrates the huge differences between the sphere/ torus case and the hyperbolic  case. The curvature of the surface being negative somehow ``filters" out almost all lower frequency oscillations of eigenfunctions over certain closed curves. In a recent paper of the second author, a first result towards the above conjecture was obtained for the case when $\gamma$ is a periodic geodesic.
\begin{theorem}[\cite{GP}]\label{GP}
	Let $\gamma=\gamma_{per}$ be a periodic geodesic on a Riemannian surface $(M,g)$ with curvature $K$ satisfying the averaged vanishing conditions in the sense of \cite{Gauss}. Given $0<c<1$,  if $\nu$ is an integer multiple of $2\pi|\gamma|^{-1}$ such that $0\le\frac{|\nu|}{\la}=\epsilon\leq c<1$,  then we have
	\begin{equation}\label{GPs}
	\Big|\int_{\gamma}e_\lambda(\gamma(s)) e^{-i\nu s}\, ds\Big|\le C |\gamma|(\log\la)^{-\frac12},
	\end{equation}
	where the constant $C$ only depends on $M$ and $c$.
\end{theorem}
It is clear that Theorem \ref{GP} is an improvement over both \eqref{period'} and \eqref{gauss}, where it is better than \eqref{period'} for negatively curved surfaces, and contains \eqref{gauss'} as the special case  when $\nu=0$. This was the first result showing generalized periods converge to 0 uniformly for all $|\nu|<c\la$ over closed geodesics on compact hyperbolic surfaces.  

The proof of Theorem \ref{GP} followed the strategies developed in \cite{CSPer} and \cite{Gauss}. The Gauss-Bonnet Theorem was used to exploit the defects of geodesic quadrilaterals which arise naturally in these
arguments. Gauss-Bonnet Theorem allows one to quantitatively avoid geodesic parallelograms on negatively curved surfaces, which in turn provides favorable controls over  derivatives of the phase functions 
which occur in the stationary phase arguments.  

The purpose of this paper is to prove log-improved generalized periods bounds for a larger class of curves, which, in particular, includes geodesic circles of arbitrary radius. Since we will not be dealing with geodesics mostly, the Gauss-Bonnet Theorem will not be as handy as in \cite{Gauss} and \cite{GP}. Instead, we shall turn to the strategies developed by the first author in \cite{emmett1} and \cite{emmett2}, and use these  to prove an analog of Theorem \ref{W} for generalized periods over curves satisfying assumptions analogous to \eqref{curvature hypotheses}. However, in this case, our curvature assumption on $\gamma$ has to involve the frequency ratio $\epsilon=\nu/\la$. Our main result is the following.
\begin{theorem}[Main Theorem]\label{main}
	Let $\gamma$ be a smooth, closed, unit-speed curve on a compact Riemannian surface $(M,g)$ with nonpositive curvature. Let $E_\gamma$ denote the set of $\epsilon \in (-1,1)$ for which
	\begin{align}
	\label{curv hyp +'} \left\langle \frac{D}{dt} \gamma', \gamma'^\perp \right\rangle &\neq - \sqrt{1 - \epsilon^2} \bfk_\gamma(\gamma'^\perp),  \qquad \text{ and } \\
	\label{curv hyp -'} \left\langle \frac{D}{dt} \gamma', \gamma'^\perp \right\rangle &\neq \sqrt{1 - \epsilon^2} \bfk_\gamma(-\gamma'^\perp)
	\end{align}
	at each point along $\gamma$.
	If $\nu$ is an integer multiple of $2\pi|\gamma|^{-1}$ such that $\nu/\lambda \in E_\gamma$, then we have
	\begin{equation}\label{periods}
	\left|\int_{\gamma}e_\lambda(\gamma(s)) e^{-i\nu s}\, ds\right| \le C (\log\la)^{-\frac12},
	\end{equation}
	where $C$ is uniform for $\nu/\lambda$ in a compact subset of $E_\gamma$.
\end{theorem}

\begin{remark} \label{main remark}
Fix $\gamma$ and $E_\gamma$ as in Theorem \ref{main}, and let $K \subset E_\gamma$ be compact. Note $E_\gamma$ still contains $K$ even if we perturb $\gamma$ so that the change in the first two derivatives of $\gamma$ are small. At the same time, careful observation throughout the proof of Theorem \ref{main} shows that the constant $C$ is uniform if we perturb $\gamma$ so that the change in the first $N$ derivatives of $\gamma$ are small, where $N$ is some fixed, finite number. Hence the constant $C$ in \eqref{periods} is uniform if $(\gamma, \nu/\lambda)$ belongs to a compact subset of
\[
	\{ (\gamma,\epsilon) : \gamma \in C^\infty(\R,M), \epsilon \in E_\gamma \} \subset C^\infty(\R,M) \times (-1,1)
\]
with respect to the subspace topology.
\end{remark}

On the flat torus $\mathbb T^2$, \eqref{curv hyp +'}-\eqref{curv hyp -'} are valid for any closed curves with non-vanishing curvature, since $\mathbf k_p(v)\equiv0$ in this case. On the other hand, for hyperbolic surfaces with curvature $K=-1$, it is known that $\mathbf k_p(v)\equiv1$, thus Theorem \ref{main} implies that any curves with curvature bounded away from $\sqrt{1-\epsilon^2}$ will enjoy log-improved generalized periods bounds for $\nu=\epsilon\la$. An direct yet significant corollary of our main theorem gives the first result towards Conjecture \ref{C2} in the case of geodesic circles. It follows from the fact that for any given geodesic circle on a Riemannian surface of nonpositive curvature always satisfies \eqref{curv hyp +'}-\eqref{curv hyp -'}, and therefore, Theorem \ref{main} and Remark \ref{main remark} implies the following.

\begin{corr} \label{corollary}
	Let $\gamma$ be a unit-speed geodesic circle on a compact Riemannian surface $(M,g)$ with nonpositive curvature. If $\nu$ is an integer multiple of $2\pi |\gamma|^{-1}$, then given any $0<\delta<1$, we have
	\[
	\left| \int_\gamma e_\lambda(\gamma(s)) e^{-i\nu s} \, ds \right| \leq C(\log \lambda)^{-1/2},
	\]
	where the constant $C$ is uniform over the sets of all geodesic circles $\gamma$ with radii in $[\delta, \delta^{-1}]$ and all $\nu$ with $|\nu|/\lambda$ in $[0, 1 - \delta]$.
\end{corr}

 Note that the number $\delta$ in Corollary \ref{corollary} can be taken to be arbitrarily close to 0, and thus this log-improved bounds indeed hold for any fixed geodesic circle with arbitrary radius.

Another corollary is about the weak $L^2(\gamma)$ limit of eigenfunctions restricted to curves. 

\begin{corr} \label{co}	Let $\gamma$ be a curve on a compact Riemannian surface $(M,g)$ with nonpositive curvature that satisfies the assumption of Theorem \ref{main}, with $0$ being an interior point of $E_\gamma$, e.g. a geodesic circle.  Then  for any orthonormal sequence of eigenfunctions $\{e_{\la_j}\}$ with frequency $\la_j$, $e_{\la_j}|_\gamma\rightarrow 0$ weakly in $L^2(\gamma)$ if and only if it is bounded in $L^2(\gamma)$.
\end{corr} 

The quantum ergodic restriction (QER) theorem of Toth and Zelditch \cite{TZ} implies that for a compact hyperbolic surface, any orthonormal sequence of eigenfunctions will have a full density subsequence that has bounded $L^2(\gamma)$ normal over a given geodesic circle $\gamma$. Therefore Corollary \ref{co} implies the following.

\begin{corr} Let $(M,g)$ be a compact hyperbolic surface, $\{e_{\la_j}\}$ an orthonormal basis of eigenfunctions of frequency $\la_j$. Then given a geodesic circle $\gamma$, there exists a density one subsequence $ \{e_{\la_{j_k}}\}$ such that $ e_{\la_{j_k}}|_\gamma\rightarrow 0$ in weak $L^2(\gamma)$.
\end{corr} 

Our paper is organized as follows.  In the next section we shall perform several standard reductions by using the reproducing kernels for eigenfunctions and lifting the calculations to the universal cover $(\mathbb R^2,\widetilde g)$. By microlocally decomposing the measure of $\gamma$ into one tangential and two transversal components, we reduce the proof of Theorem \ref{main} to estimating a few microlocalized
oscillatory integrals over curves in $(\mathbb R^2,\widetilde g)$. In \S 3,  we recall a stationary phase technique from \cite{HI}, and then use it to handle the term which is local in time. In \S 4, we gather a few bounds for the non-local kernel which are from \cite{Gauss}, and then use them to handle the tangential term.
In \S 5 we prove a few phase function bounds by employing strategies from \cite{emmett2}, and then finish the proof of Theorem \ref{main} by proving favorable bounds for the two transversal parts.
The proof for the corollaries can be found at the end of \S 5.

In what follows, by possibly rescaling the metric, we shall assume that the injectivity radius of $M$, $\text{Inj}\,M$, is at least 10. We shall always use the letter $\epsilon$ to denote a number in $(-1,1)$ that equals the frequency ratio $\nu/\la$. The letter $C$ will be used to denote various positive constants depending on $(M,g)$ and $\delta$, whose value could change from line to line.

\textit{Acknowledgments.}
The authors would like to thank Professor Sogge for his constant support. The second author want to thank Professor Greenleaf and Iosevich for their invaluable mentorship.

%%%%%%%%%%%%%%%%%%%%%%%%%%%%%%%%%%%%%%%%%%%%
% STANDARD REDUCTION AND Microlocal Decomp %
%%%%%%%%%%%%%%%%%%%%%%%%%%%%%%%%%%%%%%%%%%%%

\section{Standard Reductions and Microlocal Decompositions }

Since we are taking $\epsilon$ to be in a compact subset of $E_\gamma$, we may assume there exists a small constant $\delta > 0$ such that
\begin{align}
	|\epsilon| &\leq 1 - \delta,
\end{align}
\begin{align}
	\label{curv hyp +} \left|\left\langle \frac{D}{dt} \gamma', \gamma'^\perp \right\rangle + \sqrt{1 - \epsilon^2} \bfk_\gamma(\gamma'^\perp) \right| &\geq \delta,  \qquad \text{ and } \\
	\label{curv hyp -} \left|\left\langle \frac{D}{dt} \gamma', \gamma'^\perp \right\rangle - \sqrt{1 - \epsilon^2} \bfk_\gamma(-\gamma'^\perp) \right| &\geq \delta
\end{align}

By using a partition of unity on $\R/|\gamma|\Z$ and the triangle inequality, we can obtain \eqref{periods} by showing
\begin{equation} \label{main2 bound}
	\left| \int b(t) e_\lambda(\gamma(t)) e^{-i\nu t} \, dt \right| \leq C (\log \la)^{-1/2}
\end{equation}
where $b$ is a smooth function on $\R$ with small support. To begin we assume the support of $b$ is contained in some unit interval in $\R$, though we may further restrict the support of $b$ as needed.

% Uniformization
Let us choose a function $\rho\in {\mathcal S}(\R)$ satisfying
$$\rho(0)=1 \quad \text{and } \, \, \Hat \rho(\tau)=0 , \quad |\tau|\ge 1/4,$$
and for any $T>0$ define the Fourier multiplier operator $\rho(T(\la-\sqrt{-\Delta_g}))$ by the spectral theorem, i.e.
\[
\rho(T(\la - \sqrt{-\Delta_g})) = \sum_j \rho(T(\la - \la_j)) E_j
\]
where $E_j$ is the orthogonal projection operator onto the eigenspace spanned by $e_j$.
$\rho(T(\la - \sqrt{-\Delta_g}))$ reproduces eigenfunctions, in the sense that $\rho(T(\la-\sqrt{-\Delta_g}))e_\la=e_\la$. \eqref{main2 bound} will follow from the stronger\footnote{\eqref{main2 bound'} implies Theorem \ref{main} holds for $L^2$-normalized quasimodes which have spectral support on bands $[\la, \la + 1/\log \la]$ of length $1/\log \la$.} bound
\begin{equation} \label{main2 bound'} \tag{\ref{main2 bound}$'$}
\left|\, \int b(t) e^{-i\nu t} \rho(T(\lap g - \la))f(\gamma(t))\, dt \, \right|\le C \, (\log\la)^{-1/2}\, \|f\|_{L^2(M)},
\end{equation}
where
\begin{equation} \label{T = clog}
T = c\log \lambda
\end{equation}
for some sufficiently small $c$.

% Isolating the directions
Choose Fermi local coordinates $x = (x_1,x_2)$ about $\gamma$, so that $x_1 \mapsto (x_1,0)$ parametrizes $\gamma$ and $x_2 \mapsto (x_1,x_2)$ are geodesics normal to $\gamma$. By construction,
\begin{equation}\label{fermi metric}
g = \begin{bmatrix} 1 & 0 \\ 0 & 1 \end{bmatrix} \qquad \text{ for } x_2 = 0.
\end{equation}
Let $B_1$ be a smooth function on $S^1$ taking values in the range $[0,1]$, and for which
\begin{align*}
B_1(\xi) &= 1 \qquad \text{ for } \xi_2 \geq \delta/2 \text{ and } \\
B_1(\xi) &= 0 \qquad \text{ for } \xi_2 \leq \delta/4
\end{align*}
where here $\xi = (\xi_1,\xi_2) \in S^1$. Set
\[
B_{-1}(\xi) = B_1(-\xi) \qquad \text{ and } \qquad B_0 = 1 - B_1 - B_{-1}
\]
(see Figure \ref{fig2}).
\begin{figure}
	\centering
	\includegraphics[width=.65\textwidth]{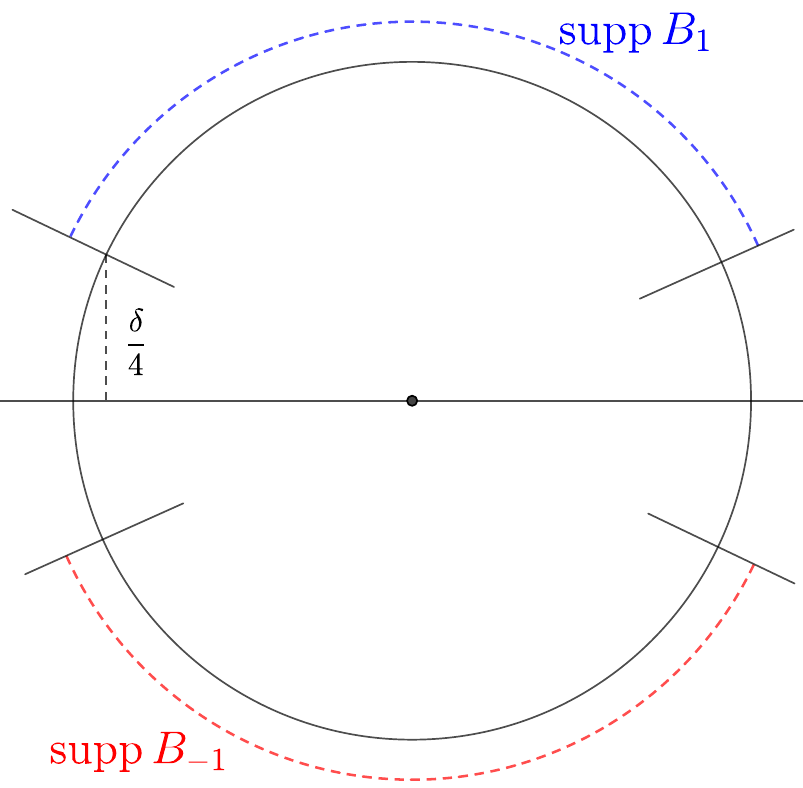}
	\caption{}
	\label{fig2}
\end{figure}
For $i = -1,0,1$, we define
\begin{equation} \label{def B}
B_i(x,y,\xi) = b_*(x) \beta(|x - y|) B_i(\xi/|\xi|)  \Upsilon_{c_1}(|\xi|/\la)
\end{equation}
where  $x$ and $y$ are expressed in our Fermi local coordinates,
and we have taken $b_*$ to be a smooth function supported on a neighborhood of $\gamma$ such that $b_*(\gamma(s))=b(s)$; $\beta \in C^\infty(\R,[0,1])$ with $\beta \equiv 1$ on a neighborhood of $0$ and $\supp \beta$ small; given constant $0<c_1<1$, $\Upsilon_{c_1}\in C^\infty(\mathbb R)$ satisfies
\begin{equation}
\Upsilon_{c_1}(r)=1, \quad r\in[c_1,c_1^{-1}],\quad \Upsilon_{c_1}(r)=0, \quad r\not\in[c_1/2,2c_1^{-1}]
\end{equation}
We associate operators\footnote{The function $B_i(x,y,\xi)$ will serve the same purpose as a zero-order symbol in $\xi$, and their respective operators will similarly function like zero-order classical pseudodifferential operators. It is important to distinguish these objects from symbols and pseudodifferential operators due to the presence of the cutoff away from $|\xi| = \lambda$, even though they play the same roles.} with $B_i$ by
\[
B_if(x) = \frac{1}{(2\pi)^2} \int_{\R^2} \int_{\R^2} e^{i\langle x - y , \xi \rangle} B_i(x,y,\xi) f(y) \, dy \, d\xi
\]
in our local coordinates. First, note that $B_i$ is a bounded operator on $L^\infty$ with norm
\begin{equation} \label{B_i norm}
\| B_i \|_{L^\infty\rightarrow L^\infty} = O(\lambda^2)
\end{equation}
for $i = -1,0,1$.
Secondly, note
\[
B_1 + B_{-1} + B_0 = 1-B_\#,
\]
where $B_\#$ is a psedudodifferential with symbol supported away from the set $\{(x,y,\xi):||\xi|/\la\in[c_1,c_1^{-1}]\}$. As in \cite[Page 141]{SFIO}, one can then use a parametrix for the half-wave operator to see that if $c_1$ is small enough, $T=c\log\la$, we have
\begin{equation} \label{Bsharp}
\|B_\#\circ \rho(T(\lap g - \la))\|_{L^2\rightarrow L^\infty}\le C_N\la^{-N}, \quad \text{for any }N\in\mathbb N,
\end{equation}
Indeed, using the H\"ormander parametrix for the half wave operator, we can see that the highest order term of the kernel associated to the above operator is 
\begin{multline*}
K(T,\la;x,y) =\iiiint e^{i(\varphi(z,y,\xi)  +\langle x - z ,\zeta \rangle  - \tau( p(z,\xi) - \la ))} \\ a(T, \lambda; \tau, z, y, \xi)B_{\#}(x,z,\zeta) \,
dz\,d\xi \, d\zeta\,d\tau.
\end{multline*}
where $p(y,\xi) $
is the principal symbol of $\lap g$; $\varphi\in C^\infty(\R^n\setminus\{0\})$ is homogeneous of degree $1$ in $\xi$, and satisfies
\begin{equation} \label{local varphi}
|\partial_{\xi}^\alpha (\varphi(x,y,\xi) - \langle x - y, \xi \rangle)| \leq C_\alpha |x - y|^2|\xi|^{1 - |\alpha|},
\end{equation}
for multiindices $\alpha \geq 0$ and for $x$ and $y$ sufficiently close;
the symbol $a$ behaves like a zero order symbol in $\xi$. For a more detailed description of the  H\"ormander paramterix, see section 3 or ~\cite{SFIO}.

Notice that $p(y,\xi)\sim|\xi|$, then if $|\xi|\not\in[c\la,c^{-1}\la]$ for some suitable constant $c$ depending on the metric, we can integrate by parts in $\tau$ to see that
\[\left|\int e^{i \tau( p(y,\xi) - \la )} a(T, \lambda; \tau, z, y, \xi) \,d\tau\right|\le C_N(|\xi|+\la)^{-N},\]
thus the difference between $K(T,\la;x,y)$ and 
\begin{multline*}
\widetilde K(T,\la;x,y) =\iint e^{i(\varphi(z,y,\xi)  +\langle x - z ,\zeta \rangle  - \tau( p(y,\xi) - \la ))}\\  \Upsilon_{c}(|\xi|/\la) q(T, \lambda; \tau, z, y, \xi)B_{\#}(x,z,\zeta) \,
dz\,d\xi \, d\zeta\,d\tau,
\end{multline*}
satisfies the required bound if $c_1$ is sufficiently small. It is then clear $\widetilde K$ gives an operator satisfying bounds in \eqref{Bsharp} if we integrate by parts in $z$ a few times. Therefore, by \eqref{Bsharp}, we have
\begin{equation} 
\left|\, \int b(t) e^{-i\nu t} B_\# \rho(T(\lap g - \la))f(\gamma(t))\, dt \, \right|\le C_N \, \la^{-N} \|f\|_{L^2(M)},\quad \text{for any }N\in\mathbb N.
\end{equation}

To prove \eqref{main2 bound'}, it suffices to show that
\begin{equation} \label{main2 bound''} \tag{\ref{main2 bound}$''$}
\left|\, \int b(t) e^{-i\nu t} B_i \rho(T(\lap g - \la))f(\gamma(t))\, dt \, \right|\le C \, (\log\la)^{-1/2}\, \|f\|_{L^2(M)},
\end{equation}
for each $i = -1,0,1$.
% Cauchy-Schwarz and orthogonality
To set up the proof of \eqref{main2 bound''} we first note that the kernel of the operator thereof is given by
$$
B_i\rho(T(\lap g - \la))(x,y)=\sum_j \rho(T(\la_j-\la)) B_ie_j(x)\overline{e_j(y)}.
$$
Hence, by the Cauchy-Schwarz inequality, we would have \eqref{main2 bound'} if we could show that
$$
\int_M\left|\, \int \,e^{-i\nu t} \sum_j \rho(T(\la_j-\la)) \, B_i e_j(\gamma(t)) \overline{e_j(y)} \, dt \, \right|^2 \, dV_g(y)
\le C(\log\la)^{-1},
$$
By orthogonality, if $\chi(\tau)=|\rho(\tau)|^2$, this is equivalent to showing that
\begin{equation}\label{2.2}
\left|\, \iint e^{i\nu(s-t)}\sum_j \chi(T(\la_j-\la)) B_i e_j(\gamma(t)) \overline{B_i e_j(\gamma(s))} \, dt ds\, \right|
\le C(\log \la)^{-1}.
\end{equation}
By Fourier inversion,
\begin{align*}
\sum_j \chi(T(\la_j-\la))B_ie_j(x)\overline{B_ie_j(y)} &= \frac{1}{2\pi T} \sum_j \int \hat \chi(\tau/T) e^{-i\tau \la} e^{i\tau \la_j} B_ie_j(x)\overline{B_ie_j(y)} \, d\tau \\
&= \frac1{2\pi T} \int \hat \chi(\tau/T) e^{-i\tau \la} 
\bigl(B_i e^{i\tau \sqrt{-\Delta_g}} B_i^*\bigr)(x,y) \, d\tau,
\end{align*}
and \eqref{2.2} is equivalent to
\begin{equation} \label{2.3}
\frac{1}{T} \left| \iiint \hat \chi(\tau/T) e^{i\nu(s-t)} e^{-i\tau \lambda} (B_i e^{i\tau \lap g} B_i^*)(\gamma(t), \gamma(s)) \, ds \, dt \, d\tau \right| \leq C (\log \la)^{-1}.
\end{equation}

We cut the integral into $\beta(\tau)$ and $(1 - \beta(\tau))$ components where, as before, $\beta \in \C_0^\infty(\R,[0,1])$ is a bump function with $\beta(\tau) = 1$ for $|\tau| \leq 1$ and $\beta(\tau) = 0$ for $|\tau| \geq 2$. We bound the $\beta(\tau)$ component first. 

\begin{proposition}\label{beta component}
	\begin{equation}\label{local term}
	\frac{1}{T} \left| \iiint \beta(\tau) \hat \chi(\tau/T) e^{i\nu(s-t)} e^{-i\tau \lambda} (B_i e^{i\tau \lap g} B_i^*)(\gamma(t), \gamma(s)) \, ds \, dt \, d\tau \right| \leq CT^{-1}.
	\end{equation}
	for $i = -1,0,1$.
\end{proposition}
The proof of Proposition \ref{beta component} is quite standard yet a bit involved. We will give a detailed proof in the next section with the help of a stationary phase lemma from \cite{HI}.

What remains is to show
\begin{align}
\nonumber \frac{1}{T} \left| \iiint (1 - \beta(\tau)) \hat \chi(\tau/T) e^{i\nu(s-t)} e^{-i\tau \lambda} (B_i e^{i\tau \lap g} B_i^*)(\gamma(t), \gamma(s)) \, ds \, dt \, d\tau \right| &\\
\label{precosine} &\hspace{-4em} \leq C (\log \la)^{-1}.
\end{align}
To deal with \eqref{precosine}, we lift to the universal cover.
% Cosine
Before this, however, we want to replace the operator $e^{i\tau \lap g}$ with $\cos(\tau \lap g)$ so we can make use of H\"uygen's principle. By Euler's formula,
\[
e^{i\tau \lap g} = 2\cos(\tau \lap g) - e^{-i\tau \lap g},
\]
and so the integral in \eqref{precosine} is
\begin{align*}
&\frac{2}{T} \iiint (1 - \beta(\tau)) \hat \chi(\tau/T) e^{i\nu(s-t)} e^{-i\tau \lambda} (B_i \cos(\tau \lap g) B_i^*)(\gamma(t), \gamma(s)) \, ds \, dt \, d\tau \\
&+ \frac{1}{T} \iiint (1 - \beta(\tau)) \hat \chi(\tau/T) e^{i\nu(s-t)} e^{-i\tau \lambda} (B_i e^{-i\tau \lap g} B_i^*)(\gamma(t), \gamma(s)) \, ds \, dt \, d\tau.
\end{align*}
By reversing our reduction, the second line is
\begin{equation}\label{+term}
\frac{1}{T} \sum_j X_T(-(\la_j + \la)) \left| \int e^{-i\nu t} B_i e_j(\gamma(t)) \, dt \right|^2
\end{equation}
where $\hat X_T(\tau) = (1 - \beta(\tau)) \hat \chi(\tau/T)$.
We claim \eqref{+term} is $O(\lambda^{-N})$ uniformly for $T \geq 1$ for each $N = 1,2,\ldots$. Indeed, \eqref{B_i norm} and the general sup-norm estimates
\[
\|e_j\|_{L^\infty(M)} \leq C \lambda_j^{1/2}
\]
(see for example ~\cite{SFIO}) imply the integral in \eqref{+term} is bounded by $C\lambda_j^{1/2}\lambda^{2}$. The $O(\lambda^{-N})$ bound follows since $X_T(-(\lambda_j + \lambda))$ is rapidly decaying in $\la_j + \la$.
%Indeed, we first  note that as a 0-order pseudodifferential operator,  $B_i$ is bounded on $H^s$ for any $s$. Secondly, since $e_j$ has $H^s$ norm $\sim\la_j^s$, Sobolev embedding implies that $\sup_M |B_i e_j|=O(\la_j^{2})$. Finally, since  $\chi(T(\la_j + \la))$ is rapidly decaying in $\la_j + \la$, \eqref{+term}  follows.
It now suffices to show
\begin{align}
\nonumber	\frac{1}{T} \left| \iiint (1 - \beta(\tau)) \hat \chi(\tau/T) e^{i\nu(s-t)} e^{-i\tau \lambda} (B_i \cos(\tau \lap g) B_i^*)(\gamma(t), \gamma(s)) \, ds \, dt \, d\tau \right|&\\
\label{2.4}	 &\hspace{-5em} \leq C (\log \la)^{-1}.
\end{align}

% Lift to the Universal cover.
Here $\bigl(\cos \tau\sqrt{-\Delta_g}\bigr)(x,y)$ is the wave kernel for the map $$C^\infty(M)\ni f\to u\in C^\infty(\R\times M)$$ solving the Cauchy problem
with initial data $(f,0)$, i.e.,
\begin{equation}\label{Cauchy}\bigl(\partial_\tau^2-\Delta_g\bigr)u=0, \quad u(0, \, \cdot \, )=f, \quad \partial_\tau u(0, \, \cdot \, )=0.\end{equation}

To be able to compute the integral in \eqref{2.4} we need to relate the wave kernel on $M$ to the corresponding wave kernel on the universal cover of $M$. By the Cartan-Hadamard Theorem  (see e.g. \cite[Chapter 7]{doCarmo}), we can lift the calculations up to the universal cover $(\mathbb{R}^2,\tilde g)$ of $(M,g)$.

Let $\Gamma$ denote the group of deck transformations preserving the associated covering map ${\mathbb{R}^2} \to M$ coming from the exponential map about the point $\gamma(0)$. The metric $\tilde g$ is the pullback of $g$ via the covering map. We shall measure the distances in $(\mathbb{R}^2,\tilde g)$ using its Riemannian distance function $d_{\tilde g}({}\cdot{},{}\cdot{})$. We choose a Dirichlet fundamental domain, $D \simeq M$, centered at the lift $\tilde \gamma(0)$ of $\gamma(0)$, which has the property that $\Rt$ is the disjoint union of the
$\alpha(D)$ as $\alpha$ ranges over $\Gamma$ and $\{\tilde y\in \Rt: \, \dgt(0,\tilde y)<10\} \subset D$ since we are assuming that $\text{Inj}\,M\ge10$.   It then follows that we can identify every point $x\in M$ with the unique point $\tilde x\in D$ having the property
that $\tilde x \mapsto x$ through the covering map.  Let also $\{\tilde \gamma(t)\in\mathbb R^2: |t|\le \tfrac12\}$ similarly denote the set of points in $D$ corresponding to our segment
$\gamma$ in $M$. Derivatives of $\gamma$ correspond to push-forwards of the corresponding derivatives of $\tilde \gamma$ through the covering map. In particular, $\tilde \gamma$ is a unit-speed curve with the same geodesic curvature as $\gamma$.
Moreover, if $p \in M$ and $v \in S_pM$, and $\tilde p$ and $\tilde v$ are their respective lifts to $D$ and $T_{\tilde p}D$, then $\tilde \bfk_{\tilde p}(\tilde v)$ as defined on the universal cover coincides with $\bfk_p(v)$. Hence, the hypotheses \eqref{curv hyp +} and \eqref{curv hyp -} correspond exactly to
\begin{align*}
\left| \left\langle \frac{D}{dt} \tilde \gamma', \tilde \gamma'^\perp \right\rangle + \sqrt{1 - \epsilon^2} \tilde \bfk_{\tgamma}(\tgamma'^\perp) \right| &\geq \delta, \qquad \text{ and } \\
\left| \left\langle \frac{D}{dt} \tgamma', \tgamma'^\perp \right\rangle - \sqrt{1 - \epsilon^2} \tilde \bfk_{\tgamma} (-\tgamma'^\perp) \right| &\geq \delta.
\end{align*}
Finally, if $\Delta_{\tg}$ denotes the Laplace-Beltrami operator associated to $\tg$, then, since solutions of the  Cauchy problem \eqref{Cauchy} for $(M,g)$ correspond exactly
to  $\Gamma$-invariant solutions of the corresponding Cauchy problem associated to the lifted wave operator $\partial^2_t-\Delta_{\tg}$, we have the following
Poisson formula relating the wave kernel on $(M,g)$ to the one for the universal cover $(\Rt,\tg)$:
\begin{equation}\label{2.10}
\bigl(\cos \tau \sqrtg \big)(x,y)=\sum_{\alpha\in \Gamma}\bigl(\cos \tau \sqrtd  \bigr)(\tilde x, \alpha(\tilde y)).
\end{equation}
Then,
\begin{align*}
&B_i \cos(\tau \lap g) B_i^*(x,y) \\
&= \frac{1}{(2\pi)^4} \iiiint e^{i\langle x - w, \eta \rangle } B_i(x,w,\eta) \cos(\tau \lap g)(w,z) e^{i\langle z - y, \zeta \rangle} \overline{B_i(z,y,\zeta)}\\
&\hspace{28em} \, dw \, dz \, d\eta \, d\zeta \\
&= \frac{1}{(2\pi)^4} \sum_{\alpha \in \Gamma} \iiiint e^{i\langle \tilde x - \tilde w, \eta \rangle } \widetilde B_i(\tilde x,\tilde w,\eta) \cos(\tau \lap \tg)(\tilde w,\alpha (\tilde z)) e^{i\langle \tilde z - \tilde y, \zeta \rangle} \overline{ \widetilde B_i(\tilde z,\tilde y,\zeta)}\\
&\hspace{29em} \, d\tilde w \, d\tilde z \, d\eta \, d\zeta
\end{align*}
where
\[
\widetilde B_i(\tilde x, \tilde y,\xi) =
\begin{cases}
B_i(x,y,\xi), & {\text{if }}{\tilde x, \tilde y \in D,} \\
0, & \text{otherwise,}
\end{cases}
\]
whereafter we write
\[
B_i \cos(\tau \lap g) B_i^*(x,y) = \sum_{\alpha \in \Gamma} \widetilde B_i \cos(\tau, \lap \tg)  \widetilde B_{i,\alpha}^{*}(\tilde x, \tilde y),
\]
where, in terms of the kernel, $$\widetilde B_{i,\alpha}(\tilde x,\tilde y)=\widetilde B_i(\alpha^{-1}(\tilde x),\tilde y).$$ 
\eqref{2.4} will follow from
\begin{equation} \label{lifted bound}
\frac{1}{T} \sum_{\alpha \in \Gamma} \left| \iint e^{i\nu(s - t)} K_{i,\alpha} (T,\lambda;\tilde \gamma(t), \tilde \gamma(s)) \, ds \, dt \right| \leq C (\log \lambda)^{-1}
\end{equation}
where
\begin{equation} \label{kernel def}
K_{i,\alpha}(T,\lambda;x,y) = \int {(1 - \beta(\tau))} \hat \chi(\tau/T) e^{-i\tau \la} \widetilde B_i \cos(\tau \lap \tg) \widetilde B_{i,\alpha}^* (x,y) \, d\tau.
\end{equation}

% Propositions.
The bound \eqref{lifted bound}, and hence Theorem \ref{main}, follows from the propositions below and \eqref{T = clog} for some sufficiently small $c$. The first proposition treats the identity term in the sum for $i = -1,0,$ and $1$.

\begin{proposition} \label{small time} If $\alpha=I$,  the identity deck transformation, we have
	\begin{equation}\label{2.17}
	\frac{1}{T} \left| \iint e^{i\nu(s - t)} K_{i,\alpha}(T,\lambda;\tilde \gamma(t), \tilde \gamma(s)) \, ds \, dt \right| \leq C T^{-1}.
	\end{equation}
	for $i = -1,0,1$.
\end{proposition}

The next proposition treats the remaining terms for the $i = 0$ case.

\begin{proposition} \label{nonsmall time}
	\begin{equation}\label{2.18}
	\frac{1}{T} \sum_{\alpha \in \Gamma \setminus \{I\}} \left| \iint e^{i\nu(s-t)} K_{0,\alpha}(T,\lambda; \tilde \gamma(t), \tilde \gamma(s)) \, ds \, dt \right| \leq Ce^{C'T} \lambda^{-1/2}.
	\end{equation}
\end{proposition}

Fix a constant $R \gg 1$ to be determined later and which is independent of $T$, $\lambda$, $\gamma$ and $\nu$ subject to the hypotheses of Theorem \ref{main}. We set
\begin{equation} \label{A def}
A = \{ \alpha \in \Gamma : d_{\tg}(\tilde \gamma, \alpha (\tilde \gamma)) \leq R \}.
\end{equation}
and treat the contributions of $A \setminus \{I\}$ and $\Gamma \setminus A$ to the sum in \eqref{lifted bound} separately.

\begin{proposition} \label{medium time}
	\begin{equation}\label{2.20}
	\frac{1}{T} \sum_{\alpha \in A \setminus \{I\}} \left| \iint e^{i\nu(s - t)} K_{i,\alpha}(T,\lambda;\tilde \gamma(t), \tilde \gamma(s)) \, ds \, dt \right| \leq CT^{-1}
	\end{equation}
	for $i = -1,1$.
\end{proposition}

\begin{proposition} \label{large time}
	\begin{equation}\label{2.21}
	\frac{1}{T} \sum_{\alpha \in \Gamma \setminus A} \left| \iint e^{i\nu(s - t)} K_{i,\alpha}(T,\lambda;\tilde \gamma(t), \tilde \gamma(s)) \, ds \, dt \right| \leq Ce^{C'T} \lambda^{-1/2}.
	\end{equation}
	for $i = -1,1$.
\end{proposition}
As argued, Theorem \ref{main} follows from Propositions \ref{beta component} through \ref{large time}.

%%%%%%%%%%%%%%%%%%%%%%%%%%%%%%%%%%%%%%%%
% STATIONARY PHASE TOOL AND SMALL TIME %
%%%%%%%%%%%%%%%%%%%%%%%%%%%%%%%%%%%%%%%%

\section{Stationary Phase Tool and Proof of Proposition \ref{beta component}}

% Stationary phase tool

We will be dealing with oscillatory integrals with up to eight variables of integration, so it will be convenient to be able to use the method of stationary phase in stages to avoid having to work with excessively large matrices. To this end, we use ~\cite[Theorem 7.7.6 ]{HI}, summarized below for convenience.

Let $\phi(x,y)$ be a smooth phase function on $\R^m \times \R^n$ with
\[
\nabla_y \phi(0,0) = 0 \qquad \text{ and } \qquad \det \nabla_y^2 \phi(0,0) \neq 0,
\]
and let $a(\lambda;x,y)$ be a smooth amplitude with small, adjustable support satisfying
\[
|\partial_\lambda^j \partial_x^\alpha \partial_y^\beta a(\lambda;x,y)| \leq C_{j,\alpha,\beta} \lambda^{-j} \qquad \text{ for } \lambda \geq 1
\]
for $j = 0,1,2,\ldots$ and multiindices $\alpha$ and $\beta$. $\nabla_y^2 \phi \neq 0$ on a neighborhood of $0$ by continuity. There exists locally a smooth map $x \mapsto y(x)$ whose graph in $\R^m \times \R^n$ contains all points in a neighborhood of $0$ such that $\nabla_y \phi = 0$, by the implicit function theorem. Let $\sigma$ denote the signature of $\nabla_y^2 \phi$. By continuity, $\sigma$ is constant on a neighborhood of $0$. We adjust the support of $a$ to lie in the intersection of these neighborhoods.

\begin{lemma}[\cite{HI}]\label{st phase}
	Let 
	\[
	I(\lambda; x) = \int_{\R^n} e^{i\lambda \phi(x,y)} a(\lambda; x, y) \, dy
	\]
	with $\phi$ and $a$ as above. Then for any fixed positive integer $N$, there exists $R_N(\lambda;x)$ such that
	\begin{align*}
	I(\lambda; x) = (\lambda/2\pi)^{-n/2} &|\det \nabla_y^2 \phi(x,y(x))|^{-1/2} e^{\pi i \sigma/4} e^{i\lambda \phi(x,y(x))} a(\lambda; x, y(x))\\
	&\hspace{8em}+ \lambda^{-n/2 - 1} e^{i\lambda \phi(x,y(x))} R_N(\lambda; x) + O(\lambda^{-N})
	\end{align*}
	where $R_N$ has compact support and satisfies
	\[
	|\partial_\lambda^j \partial_x^\alpha R(\lambda; x)| \leq C_{j,\alpha} \lambda^{-j} \qquad \text{ for } \lambda \geq 1,
	\]
	and where the $O(\lambda^{-N})$ term is uniform in $x$.
\end{lemma}

Inspection of the proof of ~\cite[Theorem 7.7.6]{HI} shows that all the bounds, constants, supports, and neighborhoods in Lemma \ref{st phase} are uniform and only depend on finitely many derivatives of $a$ and $\phi$. Now we are ready to prove Proposition \ref{beta component}.

\begin{proof}[Proof of Proposition \ref{beta component}] For the sake of simplicity, we shall only give the proof for the case when $i=0$, the other two cases follow from the same proof.
	By the H\"ormander's parametrix for the half-wave operator, (see \cite[Chap. 4]{SFIO}), we can write
	\[
	e^{i\tau\lap g}(x,y) = \frac{1}{(2\pi)^n} \int_{\R^n} e^{i(\varphi(x,y,\xi) + \tau p(y,\xi))} q(\tau,x,y,\xi) \, d\xi
	\]
	modulo a smooth kernel, where $p(y,\xi) $
	is the principal symbol of $\lap g$, and where $\varphi\in C^\infty(\R^n\setminus\{0\})$ is homogeneous of degree $1$ in $\xi$ and satisfies
	\begin{equation} \label{local varphi}
	|\partial_{\xi}^\alpha (\varphi(x,y,\xi) - \langle x - y, \xi \rangle)| \leq C_\alpha |x - y|^2|\xi|^{1 - |\alpha|}
	\end{equation}
	for multiindices $\alpha \geq 0$ and for $x$ and $y$ sufficiently close.
	Moreover, $q$ satisfies bounds
	\begin{equation} \label{q bounds}
	|\partial_\xi^\alpha \partial_{\tau,x,y}^\beta q(\tau,x,y,\xi)| \leq C_{\alpha,\beta} (1 + |\xi|)^{-|\alpha|},
	\end{equation}
	and where for $\tau \in \supp \beta$, $q$ is supported on a small neighborhood of $x = y$. Hence, the main term of \eqref{local term} is
	\begin{equation}
	\frac{1}{T} \iint K(T,\la;\gamma(s),\gamma(t))e^{i\nu(s-t)} \, ds \, dt,
	\end{equation}
	if the kernel is the following
	\begin{multline*}
	K(T,\la;x,y) =\idotsint e^{i(\varphi(w,z,\xi) +\langle x - w,\eta \rangle +\langle z - y ,\zeta \rangle  + \tau( p(z,\xi) - \la ))} \\ a(T, \lambda; \tau, x, y, \xi, w, z, \eta, \zeta) \,
	dw\,dz\,d\xi \, d\eta \, d\zeta\,d\tau,
	\end{multline*}
	where the amplitude is given by \[
	a(T, \lambda; \tau, x, y, \xi, w, z, \eta, \zeta) = \beta(\tau)\hat \chi(\tau/T) B_0(x,w,\eta) q(\tau, w, z,  \xi) \overline{ B_0(y,z,\zeta)},
	\]
	and satisfies
	\begin{equation} \label{q bounds}
	|\partial_\xi^{\alpha_1} \partial_\eta^{\alpha_2} \partial_\zeta^{\alpha_3} \partial_{\tau,x,y,w,z}^\beta a| \leq C_{\alpha,\beta} (1 + |\xi|)^{-|\alpha_1|}(1 + |\eta|)^{-|\alpha_2|}(1 + |\zeta|)^{-|\alpha_3|}
	\end{equation}
	It suffices to show that
	\begin{equation}\label{local goal}
	\left|\iint K(T,\la;\gamma(s),\gamma(t))e^{i\nu(s-t)} \, ds \, dt\right|\le C.
	\end{equation}
	After a change of coordinates sending $(\xi,\eta,\zeta) \mapsto (\lambda \xi,\lambda \eta,\lambda \zeta)$, we have
	\begin{multline*}
	K(T,\la;x,y) =\la^6\idotsint e^{i\la\Phi(\tau,x,y,\xi,w,z,\eta,\zeta)} \\ a(T, \lambda; \tau, x, y, \la\xi, w, z, \la\eta, \la\zeta) \,
	dw\,dz\,d\xi \, d\eta \, d\zeta\,d\tau,
	\end{multline*}
	After fixing $\tau,x,y,\xi$, the phase function
	\[
	\Phi(\tau,x,y,\xi,w,z,\eta,\zeta) = \varphi(w,z,\xi) + \langle x - w,\eta \rangle + \langle z - y ,\zeta \rangle + \tau( p(z,\xi) - 1 )
	\]
	has a unique critical point in the four variables  $(w,z,\eta,\zeta)$ at
	\[
	(w,z,\eta,\zeta) = (x,y,\nabla_x \varphi(x,y,\xi), -\nabla_y \varphi(x,y,\xi) + \tau \nabla_y p(y,\xi)).
	\]
	It is easy to see that at this critical point, the Hessian of $\Phi$ is
	\[
	\nabla_{w,z,\eta,\zeta}^2 \Phi = \begin{bmatrix}
	* & * & -I & 0 \\
	* & * & 0 & I \\
	-I & 0 & 0 & 0 \\
	0 & I & 0 & 0
	\end{bmatrix},
	\]
	which has determinant $-1$ and signature $0$. Then by Lemma \ref{st phase}, we see that
	\begin{equation*}
	K(T,\lambda; x,y) = \lambda^2 \iint \widetilde a(T,\lambda; \tau, x, y, \xi) e^{i\lambda \widetilde\Phi(\tau, x, y, \xi)} \, d\tau \, d\xi,
	\end{equation*}
	modulo lower order terms. Here
	\[
	\widetilde a(T,\lambda; \tau, x, y, \xi) = a(T, \lambda; \tau, x, y, \xi, x,y,\nabla_x \varphi(x,y,\xi), -\nabla_y \varphi(x,y,\xi) + \tau \nabla_y p(y,\xi))
	\]
	and
	\[
	\widetilde\Phi(\tau,x,y,\xi) = \varphi(x,y,\xi) + \tau( p(y,\xi) - 1 ).
	\]
	Let
	\[
	\Psi(\tau,s,t,\xi) = \widetilde\Phi(\tau,\gamma(s),\gamma(t),\xi)+\epsilon(s-t),
	\]
	then we can see that the main term of \eqref{local goal} is
	\begin{equation*}
	\iiiint \widetilde a(T,\lambda; \tilde \gamma(s), \tilde \gamma(t)) e^{i\lambda \Psi(\tau,s,t,\xi)} \, ds \, dt \, d\xi \, d\tau.
	\end{equation*}
	The gradient of the phase function $\Psi$ is
	\[
	\nabla_{\tau,s,\xi_1,\xi_2} \Psi = \begin{bmatrix}
	|\xi| - 1 \\
	\epsilon + \xi_1 + O(|s - t||\xi|)\\
	s - t + O(|s-t|^2) + \tau \partial_{\xi_1} |\xi|\\
	\tau \partial_{\xi_2} |\xi| + O(|s - t|^2)\\
	\end{bmatrix},
	\]  with critical points at $(\tau, s, \xi_1, \xi_2) = (0,t,-\epsilon, \pm \sqrt{1 - \epsilon^2})$. The Hessian at these critical points is:
	\[
	\nabla_{\tau, s, \xi_1, \xi_2}^2 \Psi = \begin{bmatrix}
	0 & 0 & -\epsilon & \pm\sqrt{1-\epsilon^2} \\
	0 & * & 1 & 0 \\
	-\epsilon & 1 & 0 & 0 \\
	\pm\sqrt{1-\epsilon^2} & 0 & 0 & 0
	\end{bmatrix}
	\]
	Since by our assumption, $1 - \epsilon \geq \delta$, $\sqrt{1-\epsilon^2}$ is bounded away from zero, this matrix has determinant uniformly bounded away from zero. Then if we invoke the method of stationary phase again, \eqref{local goal} follows, finishing the proof of Proposition \ref{beta component}.
\end{proof}

%%%%%%%%%%%%%%%%%
% Kernel Bounds %
%%%%%%%%%%%%%%%%%

\section{Kernel Bounds and Proofs of Propositions \ref{small time} and \ref{nonsmall time}}

We shall need an explicit expression for the kernel
\begin{equation}
K(T,\la;x,y)=\int \bigl(1-\beta(\tau)\bigr)\hat \chi(\tau/T) e^{i\tau \la} \bigl(\cos \tau \Pe\bigr)(x,y) \, d\tau,
\end{equation}
evaluating at $(x,y)=(\tilde \gamma(t),\alpha(\tilde \gamma(s)))$. Note that $K_{i,\alpha}$ is the kernel corresponding to $\widetilde B_i$-conjugation of $K$.
The following proposition is adapted from ~\cite[Proposition 5.1]{Gauss} and characterizes the kernel $K(T,\la;x,y)$. In what follows, we let $\Delta_x$ and $\Delta_y$ denote the Laplace-Beltrami operators on $\tilde M$ operating in the $x$ and $y$ variables, respectively.

\begin{proposition}[{\cite[Proposition 5.1]{Gauss}}]\label{prop5.1}  Let $T=c\log\la$. If $d_{\tilde g}\ge1$  and $\la\gg 1$, we have
	\begin{equation}\label{kernel1}
	K(T,\la;x,y)= \la^{1/2}\sum_\pm a_\pm(T,\la; x,y) e^{\pm i\la d_{\tilde g}(x,y)} +R(T,\la,x,y),
	\end{equation}
	where
	\begin{equation}\label{kernel2}
	|a_\pm(T,\la; x,y)|\le C,
	\end{equation}
	and
	if $\ell,m=1,2,3,\dots$ are fixed
	\begin{equation}\label{kernel3}
	\Delta^\ell_x \Delta_y^m a_\pm(T,\la; x,y) = O(\exp(C_{\ell,m} d_{\tilde g}(x,y)))
	\end{equation}
	and
	\begin{equation}\label{kernel5}
	|R(T,\la,x,y)|\le\la^{-5},
	\end{equation}
	provided the constant $c>0$  is sufficiently small.
	Also,  in this case we also have
	\begin{equation}\label{kernel6}
	K(T,\la;x,y)=O(\la^{-5}), \quad \text{if } \, \, \, d_{\tilde g}(x,y)\le 1.
	\end{equation}
\end{proposition}
We remark that the bounds \eqref{kernel3}, \eqref{kernel5}, and \eqref{kernel6} are stronger than those stated in \cite{Gauss}. The first follows from the pure derivative bounds
\begin{align*}
&\Delta_x^\ell a_\pm(T,\la;x,y) = O(\exp(C_\ell d_{\tilde g}(x,y))) \qquad \text{ and }\\
&\Delta_y^\ell a_\pm(T,\la;x,y) = O(\exp(C_\ell d_{\tilde g}(x,y)))
\end{align*}
given by ~\cite[Proposition 5.1]{Gauss} and Proposition \ref{mixed derivatives prop} in the appendix.
The latter two follow easily from the proof in \cite{Gauss} by increasing the number of terms used in the Hadamard parametrix and doing integration by parts a few more times. Now one can easily see that Proposition \ref{small time} follows directly from \eqref{kernel6} and the fact that each $\widetilde B_i$ is bounded on $L^\infty$ with norm about $\la^2$.

Now we can use Lemma \ref{st phase} to compute $K_{i,\alpha}$ for $i=-1,0,1$. Indeed, from Proposition \ref{prop5.1}, we see that
\begin{multline}\label{K_i}
K_{i,\alpha}(T,\lambda;x,y) \\
= \frac{\lambda^{1/2}}{(2\pi)^4} \sum_\pm \iiiint  a_\pm(T,\lambda;w,z) e^{ i (\pm\lambda d_{\tilde g}(w,z)+\langle x - w, \eta \rangle+\langle \alpha^{-1}(z) - y, \zeta \rangle)} \\\widetilde B_i(x,w,\eta) \overline{\widetilde B_i(\alpha^{-1}(z),y,\zeta)} \, dw \, dz \, d\eta \, d\zeta 
+ B_i RB_i^*(T,\lambda; x,y).
\end{multline}
Here, the $B_iRB_i^*$ term maps $L^\infty\rightarrow L^\infty$ with norm $O(\la^{-1})$ thanks to \eqref{kernel5} and \eqref{B_i norm}. It suffices to compute the first term.

After a change of variables sending $\eta \mapsto \lambda \eta$ and $\zeta \mapsto \lambda \zeta$, we can see that the main term above is
\[
\lambda^{9/2} \iiiint a(T,\lambda; x,y,w,z,\eta,\zeta) e^{i\lambda \Phi(x,y,w,z,\eta,\zeta)} \, dw \, dz \, d\eta \, d\zeta
\]
where
\[
\Phi(x,y,w,z,\eta,\zeta) = \pm d_{\tilde g}(w,z) + \langle x - w,\eta \rangle + \langle \alpha^{-1}(z) - y , \zeta \rangle
\]
and
\[
a(T,\lambda; x,y,w,z,\eta,\zeta) = a_\pm(T,\lambda; w,z) \widetilde B_i(x,w,\lambda \eta) \overline{\widetilde B_i(\alpha^{-1}(z),y,\lambda \zeta)}.
\]
Let us first look at the gradient of the phase function $\Phi$ in all variables $(w,z,\eta,\xi)$,
\[
\nabla_{w,z,\eta,\zeta} \Phi = \begin{bmatrix}
-\eta \pm \nabla_w d_{\tilde g}(w,z) \\
\zeta \pm \nabla_z d_{\tilde g}(w,z) \\
x - w \\
\alpha^{-1}(z) - y 
\end{bmatrix},
\]
which has a unique critical point at $(w,z,\eta,\zeta) = (x,\alpha (y),\pm \nabla_x d_{\tilde g}(x,\alpha (y)), \mp \nabla_y d_{\tilde g}(x,\alpha (y)))$. In particular, if $\sigma$ is the unit-speed geodesic connecting the two points $x$ and $\alpha (y)$, with $\sigma(0) = x$ and $\sigma(d_{\tilde g}(x,\alpha(y))) = \alpha (y)$, the critical point is $(x,y,\pm \sigma'(0), \pm \alpha_* \sigma'(d_{\tilde g}(x,\alpha (y))))$, where $\alpha_*$ is the map induced on the cotangent bundle by $\alpha$. The Hessian matrix of $\Phi$ is
\[
\nabla_{w,z,\eta,\zeta}^2 \Phi = \begin{bmatrix}
* & * & -I & 0 \\
* & * & 0 & I \\
-I & 0 & 0 & 0 \\
0 & I & 0 & 0
\end{bmatrix},
\]
which has full rank. This matrix has signature $0$ and determinant $-1$ so by Lemma \ref{st phase}, modulo a $O(e^{Cd_{\tg}(x,\alpha (y))} \lambda^{-1/2})$ error, \eqref{K_i} is equal to

\begin{equation} \label{K_i'}
\lambda^{1/2} \sum_\pm a_\pm(T,\lambda;x,\alpha (y))  [b_*(x)]^2  [b_*(y)]^2 B_i(\mp \sigma'(0)) B_i(\mp \alpha_* \sigma'(d_\tg(x,\alpha(y))) e^{\pm i \lambda d_{\tilde g}(x,\alpha (y))}.
\end{equation}

% \begin{equation} \label{K_{i,\alpha}'}
% \widetilde K_{i,\alpha}(T,\lambda;x,y)=\lambda^{1/2} \sum_\pm a_\pm(T,\lambda;x,y) \widetilde B_i(x,x,\mp\sigma'(0)) \overline{\widetilde B_i(y,y,\mp \sigma'(d_{\tilde g}(x,y)))} e^{\pm i \lambda d_{\tilde g}(x,y)}.
% \end{equation}

Now we are ready to prove Propositions \ref{small time} and \ref{nonsmall time}. By \eqref{kernel6} and the same argument as before, $K_{i,\alpha}(T,\lambda;x,y) = O(\lambda^{-1})$ for $d_{\tilde g}(x,y) \leq 1$, whence follows Proposition \ref{small time}. We prove Proposition \ref{nonsmall time} below.

\begin{proof}[Proof of Proposition \ref{nonsmall time}.]
	Notice that by Huygen's Principle, the number of nonzero summands in \eqref{2.18} is at most exponential in $T$, and thus to prove Proposition \ref{nonsmall time}, it suffices to show that for each $\alpha\neq I$, there exists a constant $C$ independent of $\alpha$, such that
	\begin{equation} \label{nonsmall goal}
	\left| \iint e^{i\nu(s-t)} \widetilde K_{0,\alpha}(T,\lambda; \tilde \gamma(t), \tilde \gamma(s)) \, ds \, dt \right| \leq Ce^{CT} \lambda^{-1/2}.
	\end{equation}
	In fact, by \eqref{K_i'}, the left hand side of \eqref{nonsmall goal} is equal to
	\begin{multline} 
	\iint \lambda^{1/2} \sum_\pm a_\pm(T,\lambda;\tilde \gamma(t), \alpha(\tilde \gamma(s))) [b(t)]^2 [b(s)]^2 B_i(\mp \sigma'(0))\\ B_i(\mp \alpha_* \sigma'(r_\alpha(t,s))) e^{i\la(\epsilon(s-t) \pm  r_\alpha(t,s))} \, ds \, dt, 
	\end{multline}
	where we have set $r_\alpha(t,s) = d_\tg(\tilde \gamma(t), \alpha( \tilde \gamma(s)))$ and where $\sigma$ is the unit-speed geodesic connecting the two points $\tilde \gamma(t)$ and $\alpha (\tilde \gamma(s))$.
	The key observation here is that the phase function
	$\epsilon(s-t) \pm r_\alpha(t,s)$ has no critical point in the support of the integrand. In fact, the $(t,s)$ gradient of this phase function is
\[
	(-\epsilon \pm \partial_t r_\alpha(t,s), \epsilon\pm \partial_s r_\alpha(t,s)),
\]
	which vanishes only if the angle made by the geodesic $\sigma$ and $\tilde\gamma$ has cosine value to be equal to $\mp\epsilon$, and at the same time the angle made by $\sigma$ and $\alpha( \tilde\gamma)$ has cosine value equal to $\pm\epsilon$. Since $|\epsilon|\le 1-\delta$ is uniformly bounded away from 1, at a critical point of the phase these two angles will be uniformly bounded away from $0$ and $\pi$ independent of the choice of $\alpha$. However, neither of these can happen in the support of $a_\pm$, due to our choice of the phase support of $B_0$, see Figure \ref{fig2}. Thus, we have an absolute lower bound for the gradient of the phase function which is uniform in $\alpha$. Now we can use \eqref{kernel3} to  integrate by parts in $s$ to get \eqref{nonsmall goal}, finishing the proof of Proposition \ref{nonsmall time}.
\end{proof}

%%%%%%%%%%%%%%%%%%%%%%%%%
% Phase Function Bounds %
%%%%%%%%%%%%%%%%%%%%%%%%%

\section{Phase Function Bounds and Proofs of Propositions \ref{medium time} and \ref{large time}}

By \eqref{K_i'}, we write
\begin{multline} \label{final osc int}
	\iint e^{i\nu(s-t)} K_{i,\alpha}(T,\lambda; \tilde \gamma(t), \tilde \gamma(s)) \, ds \, dt \\
	= \lambda^{1/2} \sum_\pm \iint a_\pm(T,\lambda, \alpha;t,s) e^{i \pm \lambda \phi_\alpha(t,s)} \, ds \, dt + O(e^{CT}\lambda^{-1/2})\\
\end{multline}
with phase function\footnote{Strictly speaking, this should be $\pm\epsilon(t - s) + r_\alpha(t,s)$, however, our $\epsilon$ is allowed to be negative, so we omit the $\pm$ sign for simplicity.}
\[
\phi_\alpha(t,s) = \epsilon(t - s) + r_\alpha(t,s)
\]
and amplitude
\[
	a_\pm(T,\lambda,\alpha;t,s) = a_\pm(T,\lambda; \tilde \gamma(t), \alpha( \tilde \gamma(s))) [b(s)]^2[b(t)]^2 B_i(\mp \sigma'(0)) B_i(\mp \alpha_* \sigma'(r_\alpha))
\]
where $r_\alpha(t,s) = d_{\tilde g}(\tgamma(t), \alpha( \tgamma(s)))$, $\sigma$ is the geodesic adjoining $\tilde \gamma(t)$ and $\alpha( \tilde \gamma(s))$ as before,
and where by \eqref{kernel3} the amplitude satisfies
\begin{equation}\label{amp bound}
|\partial_t^j \partial_s^k a_\pm(T,\lambda,\alpha;t,s)| = O(e^{C_{j,k} r_\alpha}).
\end{equation}
Notice that we can control the support of the amplitude by controlling the support of $b$.

In what follows, we will only consider the case where $i = +1$. The arguments for when $i = -1$ are similar. Fix a unit normal vector field $v(t) = (0,1)$ in our Fermi coordinates along $\tgamma(t)$. Then, $B_1(\xi)$ is supported in the region $\langle v,\xi \rangle \geq \frac\delta4 |\xi|$. Hence the amplitude is supported only for those $t$ and $s$ for which
\begin{align}
&\label{supp of a} \langle \mp \sigma'(0), v(t) \rangle \geq \delta/4 \qquad \text{ and } \\
&\nonumber \langle \mp \alpha_* \sigma'(r_\alpha(t,s)), v(s) \rangle \geq \delta/4,
\end{align}
where the signs $\pm$ must match. We will use the methods of stationary and nonstationary phase to provide the desired bounds on the right side of \eqref{final osc int}, so we need some information about the first and second derivatives of $\phi_\alpha$.

Firstly, the $(t,s)$ gradient of the phase function is
\[
\nabla_{t,s} \phi_\alpha(t,s) = \begin{bmatrix}
\epsilon + \partial_t r_\alpha(t,s) \\
-\epsilon + \partial_s r_\alpha(t,s)
\end{bmatrix}.
\]
Note $\phi_\alpha$ has a critical point wherever the geodesic $\sigma$ is incident to both $\tgamma$ and $\alpha(\tgamma)$ at an angle with cosine value $\epsilon$. From ~\cite{emmett2} we have the computation
\begin{equation} \label{partial_s^2}
\partial_s^2 \phi_\alpha(t,s) = \cos(\theta) (\pm \kappa_{\tilde \gamma}(s) + \cos(\theta) \kappa_{S(\tgamma, r)}(s))
\end{equation}
where $\kappa_{S(\tgamma,r_\alpha)}(s)$ denotes the geodesic curvature at $\alpha(\tgamma(s))$ of the circle centered at $\tgamma(t)$ with radius $r_\alpha(t,s)$, and $\theta$ is the angle this circle makes with $\alpha(\tgamma)$ (see Figure \ref{fig3}). The sign of $\pm$ agrees with that of $\langle \sigma'(r_\alpha(t,s)), \frac{D}{ds} \alpha(\tgamma(s)) \rangle$, i.e. positive if $\sigma'$ agrees with the direction of curvature of $\alpha(\tgamma)$ and negative otherwise. A similar formula
\begin{equation} \label{partial_t^2}
\partial_t^2 \phi_\alpha(t,s) = \cos(\theta') (\mp \kappa_{\tgamma}(t)  + \cos(\theta')\kappa_{S(\alpha(\tgamma),r_\alpha)}(t))
\end{equation}
holds for the second derivative in $t$, where $\theta'$ is the angle the circle centered at $\alpha(\tgamma(s))$ makes with $\tgamma$ at $\tgamma(t)$, and where the sign $\mp$ \emph{disagrees} with the sign of $\langle \sigma'(0), \frac{D}{dt}\tgamma(t) \rangle$.
\begin{figure}
	\centering
	\includegraphics[width=.65\textwidth]{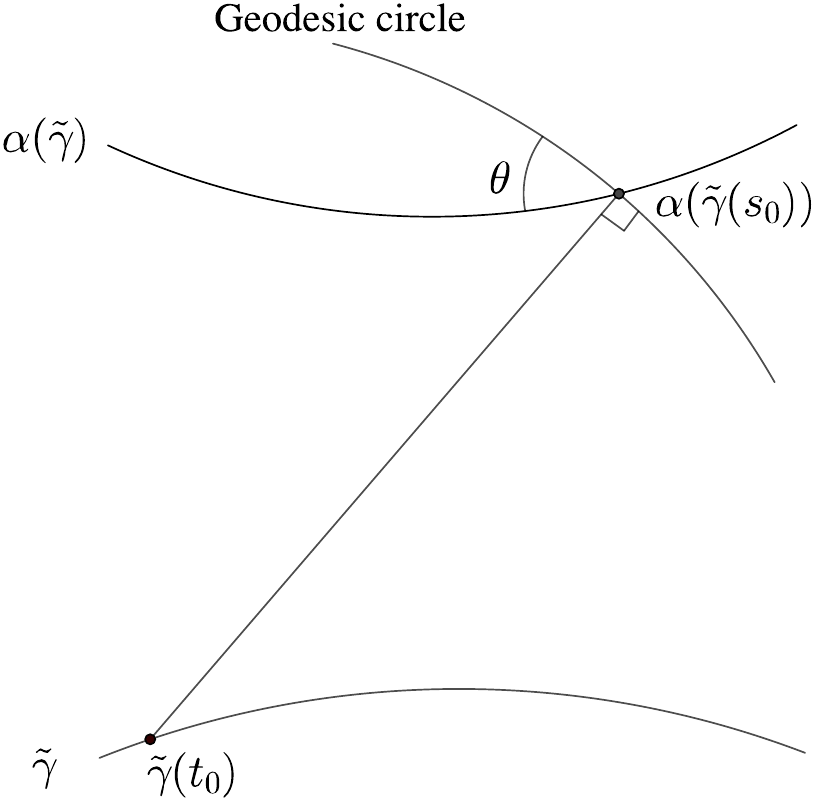}
	\caption{}
	\label{fig3}
\end{figure}
We are now in a position to prove Proposition \ref{medium time}. 

\begin{proof}[Proof of Proposition \ref{medium time}]
	Since the number of terms $\alpha \in A \setminus I$ is fixed and finite, it suffices to show
	\[
	\frac{\lambda^{1/2}}{T} \left| \iint a_\pm(T,\lambda,\alpha;t,s) e^{\pm i \lambda \phi_\alpha(t,s)} \, dt \, ds \right| \leq C_\alpha T^{-1}
	\]
	for each such $\alpha$, where the constant $C_\alpha$ is allowed to depend on $\alpha$. We claim that the union of open neighborhoods
	\[
	\{(t,s) : \nabla_{t,s} \phi_\alpha(t,s) \neq 0 \} \cup \{(t,s) : \partial_s^2 \phi_\alpha(t,s) \neq 0 \} \cup \{(t,s) : \partial_t^2 \phi_\alpha(t,s) \neq 0 \}
	\]
	covers the diagonal $t = s$ of the support of $a_\pm$. We then restrict the support of $a_\pm$ so that it lies entirely within one of these open neighborhoods. The desired bound is obtained by nonstationary phase, in the first case, or by stationary phase (e.g. by ~\cite[Corollary 1.1.8]{SFIO}) in the appropriate variable.
	
	Suppose $(t,s)$ is in the diagonal of $\supp a_\pm$ and that $\nabla_{t,s} \phi_\alpha(t,s) = 0$ and $\partial_s^2 \phi_\alpha(t,s) = 0$. Then the geodesic $\sigma$ adjoining $\tgamma(t)$ to $\alpha(\tgamma(s))$ is incident to both $\tgamma$ and $\alpha(\tgamma)$ at an angle with cosine value $\epsilon$. Moreover, $\cos(\theta) = \sqrt{1 - \epsilon^2}$ since the geodesic $\sigma$ and the circle $S(\tgamma,r_\alpha)$ intersect $\alpha(\tgamma)$ at complimentary angles. By \eqref{partial_s^2},
	\[
	0 = \pm \kappa_{\tgamma}(s) + \sqrt{1 - \epsilon^2} \kappa_{S(\tgamma,r_\alpha)}(s).
	\]
	Note this situation requires the sign $\pm$ to be negative, so that $\sigma'$ points in a direction contrary to the direction of the curvature of $\alpha(\tgamma)$. Since $s = t$ and $(t,s) \in \supp a_\pm$, $\sigma'$ must also point in a direction contrary to that of that of the curvature of $\tgamma$ (see ~\eqref{supp of a}). Therefore the sign in \eqref{partial_t^2} is positive and
	\[
	\partial_t^2 \phi_\alpha(t,s) = \sqrt{1 - \epsilon^2} \left( \kappa_{\tgamma}(t) + \sqrt{1- \epsilon^2} \kappa_{S(\alpha(\tgamma),r_\alpha)}(t) \right) > 0
	\]
	as desired.
\end{proof}

Proposition \ref{large time} requires us obtain some uniform bounds on \eqref{final osc int},  so we will need some uniform bounds on the second derivatives of $\phi_\alpha$. The bounds which follow are largely adapted from the corresponding bounds in ~\cite{emmett2}. We begin with the mixed derivative.

\begin{lemma}[{~\cite[Lemma 3.1]{emmett2}}] \label{off-diagonal lemma} We have absolute bounds
	\[
	|\partial_s \partial_t \phi_\alpha(t,s)| \leq 2/r_\alpha(t,s).
	\]
\end{lemma}

To bound the pure derivatives \eqref{partial_s^2} and \eqref{partial_t^2}, we need to be able to describe the behavior of the curvature of circles of large radius. For this we have the following lemma.
\begin{lemma}[{~\cite[Lemma 4.1]{emmett2}}] \label{large radius} Using the notation above, we have absolute bounds\footnote{Note while this lemma is stated for circles with centers along $\tgamma$, it holds for circles in general.}
	\[
	0 < \kappa_{S(\tgamma,r_\alpha)}(s) - \bfk(\sigma'(r_\alpha(t,s))) < 1/r_\alpha(t,s).
	\]
\end{lemma}

Our final geometric lemma provides bounds on the pure second derivatives of $\phi_\alpha$ and is adapted from ~\cite[Lemma 4.2]{emmett2}.

\begin{lemma} \label{diagonal lemma}
	Suppose
	\[
	|\pm \kappa_{\tilde\gamma}(s) - \sqrt{1-\epsilon^2}\kappa_{S(\tgamma,r_\alpha)}(s)| > \eps_0 \qquad \text{ for all } t,s \in \mathcal I
	\]
	for some $0<\eps_0\ll \delta $. Here, as before, the $\pm$ sign matches that of $\langle \sigma'(r_\alpha(t,s)), \frac{D}{ds} \alpha(\tgamma(s)) \rangle$. Then there exist positive constants $c_0$ and $\eta< \delta$ independent of $\alpha$ such that if the diameter of $\mathcal I$ is less than $c_0$, and $\partial_s \phi_\alpha$ is nonvanishing on $\mathcal I \times \mathcal I$, then
	\[
	|\partial_s \phi_\alpha(t,s)| \geq \eta \qquad \text{ for } t,s \in \supp b.
	\]
	On the other hand if $\partial_s \phi_\alpha(t_0,s_0) = 0$ for some $s_0,t_0 \in \mathcal I$, then
	\[
	|\partial_s^2 \phi_\alpha(t,s)| \geq \sqrt\delta\eps_0/2 \quad \text{ for } t,s \in \mathcal I.
	\]
	This result holds similarly for derivatives in $t$.
\end{lemma}

\begin{proof}
	The curvature of any geodesic circle in $(\R^2,\tilde g)$ with radius at least $1$ is bounded uniformly by Lemma \ref{large radius} and the fact that $\bfk$ is bounded ($\bfk$ is continuous on $M$). Hence, we select a global constant $C$ so that
	\[
	\sup_{s \in \mathcal I} \kappa_{\tilde\gamma}{(s)} + \sup_{t,s \in \mathcal I} \kappa_{S(\tgamma,r_\alpha)}(s) \leq C \qquad \text{ for all } \alpha \neq I
	\]
	where, as before, $\kappa_{S(\tgamma,r_\alpha)}(s)$ is the curvature of the geodesic circle at $\alpha(\tgamma(s))$, with center at $\tgamma(t)$ and radius $r=\phi_\alpha(t,s)$.
	Set
	\[
	\eta' = \min\left( \frac{1}{10}\delta, \frac{\sqrt{\delta}\,\eps_0}{4C} \right).
	\]
	%{\color{red} I think the following claim is the key difference in this generalized periods case, can you check this?} \Wedit{Made a couple adjustments. Looks good now.}
	We claim that
	\begin{equation}\label{claim}
	|\partial_s^2 \phi_\alpha| \geq \sqrt\delta\eps_0/2 \quad \text{ if } \quad |\partial_s \phi_\alpha| \leq \eta'.
	\end{equation}
	To prove this claim, first note that
	\[
	|-\epsilon+\sin(\theta)| = |\partial_s \phi_\alpha(t,s)|,
	\]
	where $\theta$ is as in Figure \ref{fig3},
	then if $|\partial_s \phi_\alpha(t,s)| \leq \eta'$, 
	\[
	|\cos (\theta) - \sqrt{1-\epsilon^2}| \leq \frac{|\sin^2 (\theta) - \epsilon^2|}{\cos (\theta) + \sqrt{1 - \epsilon^2}} \leq \frac{2|\sin (\theta) - \epsilon|}{\sqrt{\delta}} \leq \frac{\eps_0}{2C}
	\]
	Since $\eta' \leq \delta/10$, we have that $\cos(\theta) \geq \sqrt{\delta}$ by default.
	Hence,
	\begin{align*}
	|\partial_s^2 \phi_\alpha| &\geq \sqrt{\delta}\left| \pm \kappa_{\gamma} + \cos\theta \kappa \right|\\
	&= \sqrt{\delta}\left| \pm \kappa_{\gamma} + \sqrt{1-\epsilon^2}\kappa - (\sqrt{1-\epsilon^2}- \cos\theta) \kappa \right|\\
	&\geq \sqrt{\delta} | \pm\kappa_{\gamma} + \sqrt{1-\epsilon^2}\kappa | - \sqrt{\delta}|\sqrt{1-\epsilon^2} - \cos\theta| |\kappa|\\
	&\geq \sqrt{\delta} \eps_0 -\sqrt{\delta} \frac{C\eps_0}{2C}\\
	&= \frac{\sqrt{\delta}\eps_0}{2},
	\end{align*}
	proving \eqref{claim}.
	
	Now set
	\[
	c_0 = \frac{\eta'}{2\sqrt{2}(1+C^2)^{1/2}}.
	\]
	By \eqref{partial_s^2},
	\[
	|\partial_s^2 \phi_\alpha| \leq C.
	\]
	Moreover by Lemma \ref{off-diagonal lemma}, the fact that $\mathcal I$ has diameter at most $1$, and that the injectivity radius is at least $10$, we have
	\[
	|\partial_t \partial_s \phi_\alpha(t,s)| \leq 1.
	\]
	Hence for any $(t,s)$ and $(t_0,s_0)$ in $\I \times \I$,
	\begin{align*}
	|\partial_s \phi_\alpha(t,s) - \partial_s \phi_\alpha(t_0,s_0)| &\leq (1 + C^2)^{1/2}|(t,s) - (t_0,s_0)| \leq \frac{\eta'}{2}
	\end{align*}
	since the diameter of $\I \times \I$ is no greater than $\sqrt 2 c_0$.
	In particular if $\partial_s \phi_\alpha(t_0,s_0) = 0$, then
	\[
	|\partial_s \phi_\alpha(t,s)| \leq \eta'/2 \qquad \text{ for all } t,s \in \mathcal I
	\]
	and so $|\partial_s^2 \phi_\alpha(t,s)| \geq \sqrt\delta\eps_0/2$ by our claim \eqref{claim}.
	
	Now suppose $|\partial_s \phi_\alpha(t,s)| > 0$ for all $t,s \in \mathcal I$. In the case that $|\partial_s \phi_\alpha(t_0,s_0)| \leq \eta'/2$ for some $t_0,s_0 \in \mathcal I$, $|\partial_s \phi_\alpha(t,s)| \leq \eta'$ for all $t,s \in \mathcal I$, and hence by our claim, $\partial_s \phi_\alpha(t,s)$ is monotonic in $s$, and so $\partial_s \phi_\alpha$ is smallest near the endpoints of $\mathcal I$. Since $\supp b$ is closed and $\mathcal I$ open, the distance $d(\supp b, \mathcal I^c)$ from $\supp b$ to the complement of $\mathcal I$ is positive. Hence,
	\[
	|\partial_s \phi_\alpha(t,s)| \geq d(\supp b, \mathcal I^c) \sqrt\delta\eps_0/2 > 0.
	\]
	The proof is complete after setting
	\[
	\eta = \min(\eta'/2, d(\supp b, \mathcal I^c) \sqrt\delta\eps_0/2).
	\]
\end{proof}

Now we are in a position to finish the proof of Proposition \ref{large time} and hence  the proof of our main theorem.
\begin{proof}[Proof of Proposition \ref{large time}.]
	By our hypothesis \eqref{curv hyp +} and \eqref{curv hyp -} on the curvature of $\gamma$, and since $\bfk$ is continuous, we restrict the support of $b$ and the interval $\mathcal I$ so that
	
	\begin{align}\label{ca}
	\inf_{t,s \in \mathcal I} \left| \left\langle \frac{D}{dt} \gamma', \gamma'^\perp \right\rangle(t) + \sqrt{1 - \epsilon^2} \bfk_\gamma(\gamma'^\perp)(s) \right| &\geq 2\eps_0, \qquad \text{ and } \\\nonumber
	\inf_{t,s \in \mathcal I} \left| \left\langle \frac{D}{dt} \gamma', \gamma'^\perp \right\rangle(t) - \sqrt{1 - \epsilon^2} \bfk_\gamma(-\gamma'^\perp)(s) \right| &\geq 2\eps_0.
	\end{align}
	for some small $\eps_0 > 0$. We first require $R$ in \eqref{A def} be at least as large as $16\eps_0^{-1}\delta^{-1/2}$ so that, by Lemma \ref{off-diagonal lemma},
	\begin{equation} \label{phase function mixed}
	\sup_{t,s \in \mathcal I} |\partial_t \partial_s \phi_\alpha(t,s)| \leq \eps_1/8, \qquad \text{ if } \alpha \in \Gamma \setminus A,\ \eps_1=\eps_0\sqrt\delta.
	\end{equation}
	By Lemma \ref{large radius} and our requirement that $R > 16 \eps_0^{-1}\delta^{-1/2}$,
	\[
	\sqrt{1-\epsilon^2}|\kappa_{S(\tgamma,r_\alpha)}(s) - \bfk(\sigma'(r_\alpha(t,s)))|< \eps_0.
	\]
	Hence, if the $\pm$ sign matches that of $\langle \sigma'(r_\alpha(t,s)), \frac{D}{ds} \alpha(\tgamma(s)) \rangle$, it follows from \eqref{ca} that
	\[|\pm\kappa_{\tilde\gamma}(t) - \sqrt{1-\epsilon^2}\kappa_{S(\tgamma,r_\alpha)}(s)|> \eps_0, \qquad \text{ for } t,s \in \mathcal I.
	\]
	
	By proposition \ref{prop5.1} and Huygens' principle, the number of nonzero summand in  \eqref{2.21} is $O(e^{CT})$, it suffices to show
	\begin{equation}\label{final integral}
	\left| \iint a_\pm(T,\lambda,\alpha;t,s) e^{\pm i \lambda \phi_\alpha(t,s)} \, dt \, ds \right| \leq C e^{CT} \lambda^{-1},
	\end{equation}
	for all $\alpha\in\Gamma\setminus A$ and $C$ is independent of $\alpha$. 
	
	By a partition of unity, we restrict the diameter of $\mathcal I$ to be less than the constant $c_0$ in Lemma \ref{diagonal lemma}. If $|\partial_s \phi_\alpha| > 0$ on $\mathcal I \times \mathcal I$, $|\partial_s \phi_\alpha| \geq \eta$ on $\supp b \times \supp b$ for some $\eta > 0$ independent of $\alpha$. Then we integrate by parts in $s$ to see that much better bounds are satisfied in this case. We obtain \eqref{final integral} similarly if $\partial_t \phi_\alpha$ does not vanish in $\mathcal I \times \mathcal I$.
	
	What is left is the case that $\nabla \phi_\alpha$ vanishes at exactly one point $(t_0,s_0) \in \mathcal I \times \mathcal I$. By a translation, we assume without loss of generality that $(t_0,s_0) = (0,0)$. In this case,
	\begin{equation} \label{diagonal bounds}
	|\partial_s^2 \phi_\alpha| \geq \eps_1/2 \quad \text{ and } \quad |\partial_t^2 \phi_\alpha| \geq \eps_1/2
	\end{equation}
	on $\mathcal I \times \mathcal I$. Together with \eqref{phase function mixed}, it is easy to see that
	\[
		|\nabla^2\phi_\alpha(t,s) \xi| \geq \frac{c}{4} |\xi| \qquad \text{ for all } \xi \in \R^2.
	\]
	Hence by the mean value theorem,
	\[
		|\nabla \phi(t,s)| \geq \frac{c}{4}|(t,s)| \qquad \text{ for } (t,s) \in \mathcal I \times \mathcal I.
	\]
	\eqref{final integral} then follows by a standard stationary phase argument in both variables $s$ and $t$. Indeed, we gain a  $\la^{-1}$ factor from the stationary phase, and only lose by a factor of $e^{CT}$ thanks to our bounds on the amplitude \eqref{amp bound}.  
\end{proof}

\begin{proof}[Proof of Corollary \ref{corollary}.]
	The collection of circles in $M$ with radii in $[r_1,r_2]$ with $0 < r_1 < r_2 < \infty$ have uniformly bounded derivatives by compactness. Let $\gamma'^\perp$ denote the unit vector normal to such a circle $\gamma$ pointing towards its center. Note
	\[
	\left| \left\langle \frac{D}{dt} \gamma',\gamma'^\perp \right\rangle + \sqrt{1 - \eps^2} \bfk_\gamma(\gamma'^\perp) \right| > 0
	\]
	and by Lemma \ref{large radius},
	\[
	\left| \left\langle \frac{D}{dt} \gamma',\gamma'^\perp \right\rangle - \sqrt{1 - \eps^2} \bfk_\gamma(-\gamma'^\perp) \right| > 0.
	\]
	Again by compactness, these quantities are uniformly bounded away from $0$, say by $\delta > 0$. Hence, $E_\gamma = (-1 + \delta, 1 - \delta)$. The corollary follows from Theorem \ref{main}.
\end{proof}

\begin{proof}[Proof of Corollary \ref{co}.]

		 Let $\gamma$ be such a curve. As a consequence of the principle of uniform boundedness, every weakly convergent sequence in a Hilbert space is bounded. It suffices to show that any sequence of eigenfunctions with bounded $L^2(\gamma)$ restriction norms converges to $0$ weakly in $L^2(\gamma)$  
		 
		 For the sake of simplicity, we may now assume that $|\gamma|=2\pi$, and the $L^2(\gamma)$ norm of $e_{\la_j}$ is bounded by 1. Let $g\in L^2(\gamma),$ it then suffices to show that given any $\eps>0$,  for large enough $j$, we have
		\[\left|\int_\gamma e_{\la_j}(\gamma(s))\,g(s)\,ds\right|\le \eps,\]
		Now if we write $g$ in terms of its Fourier series,
		\[g(s)=\sum_ka_ke^{iks}.\]
		then $g$ being in $L^2(\gamma)$ implies that there exists a $N>0$ depending on $\eps$, such that
		\[\sum_{|k|>N}|a_k|^2\le\frac14\eps^2.\]
		If we denote
		\[b_{j,k}=\left|\int_\gamma e_{\la_j}\,e^{iks}\,ds\right|,\]
		then by Theorem \ref{main}, and the fact that $0$ is in the interior of $E_\gamma$,
		\[b_{j,k}\le C(\log\la_j)^{-\frac12},\]
		where $C$ will be an absolute constant provided that $|k|\le\delta\la_j,$ for some fixed $\delta>0$ such that $(-\delta,\delta)\subset E_\gamma$.
		
		Now we can see that	
		\begin{align*}\left|\int_\gamma e_{\la_j}\,g\,ds\right|&\le\sum_{|k|\le N}|a_k||b_{k,j}|+\Big[\sum_{|k|>N}|a_k|^2\Big]^\frac12\|e_{\la_j}\|_{L^2(\gamma)}\\
		&\le C{(\log\la_j)^{-\frac12}}\sum_{|k|\le N}|a_k|+\frac12\eps\\&\le C(\log\la_j)^{-\frac12}N^\frac12[\sum_{|k|\le N}|a_k|^2]^\frac12+\frac12\eps.\end{align*}
		
		Notice that if we take $j$ large enough, such that $\log{\la_j}\ge[{4\eps^{-2}NC^2\|g\|^2_{L^2(\gamma)}}]+1,$ the first term on the last line will be less than $\eps/2$, then we have
		\[\left|\int_\gamma e_{\la_j}\,g\,ds\right|\le \eps.\]
		This choice of $j$ can be justified, since $N/\la_j$ would be comparable to $\eps^2\la_j^{-1}\log\la_j<\delta$ when $\la_j$ is large enough, which guarantees  the uniformity of $C$.
	\end{proof}

%%%%%%%%%%%%
% APPENDIX %
%%%%%%%%%%%%

\section{Appendix}

We present here a proposition which allows us to obtain the bounds on the mixed derivatives \eqref{kernel3} from the corresponding bounds on the pure derivatives.

\begin{proposition} \label{mixed derivatives prop}
	Let $M$ and $\tilde M$ be as above and let $f \in C^\infty(\tilde M, \tilde M)$ satisfy bounds
	\begin{equation} \label{mixed derivatives prop 1}
	|\Delta_x^\ell f( x,  y)| \leq C_\ell e^{C_\ell d_{\tilde g}(x, y)} \qquad \text{ and } \qquad |\Delta_y^\ell f( x,  y)| \leq C_\ell e^{C_\ell d_{\tilde g}( x,  y)}
	\end{equation}
	for $\ell = 0,1,2,\ldots$, where $\Delta_x$ and $\Delta_y$ denote the Laplace-Beltrami operators on $\tilde M$ in the $x$ and $y$ variables, respectively. Then,
	\[
	|\Delta_x^\ell \Delta_y^m f( x,  y)| \leq C_{\ell,m} e^{C_{\ell,m} d_{\tilde g}( x,  y)} \qquad \text{ for } \ell,m = 0,1,2,\ldots
	\]
	where each of the constants $C_{\ell,m}$ depends only on $M$ and finitely many of the constants $C_\ell$.
\end{proposition}

\begin{proof}
	Let $\beta \in C_0^\infty(\R,[0,1])$ be equal to $1$ near $0$ and be supported in the interval $(-\inj M, \inj M)$. Fix $x_0,y_0 \in \tilde M$ and set
	\[
	F(x,y) = \beta(d_{\tilde g}(x,x_0)) \beta(d_{\tilde g}(y,y_0)) f(x,y).
	\]
	The support of $\beta$ allows us to interpret $F$ has a function on $M \times M$. The distance function $d_{\tilde g}$ satisfies similar bounds as \eqref{mixed derivatives prop 1}, and hence
	\begin{equation} \label{mixed bounds prop 3}
	|\Delta_x^\ell F( x,  y)| \leq C_\ell' e^{C_\ell' d_{\tilde g}(x_0, y_0)} \qquad \text{ and } \qquad |\Delta_y^\ell F( x,  y)| \leq C_\ell' e^{C_\ell' d_{\tilde g}( x_0,  y_0)}.
	\end{equation}
	for $\ell = 0,1,2,\ldots$ where the constants $C_\ell'$ depend only on $C_\ell$, $\beta$, and $M$. Moreover, it suffices to show
	\[
	\| \Delta_x^\ell \Delta_y^m F \|_{L^\infty(M \times M)} \leq C_{\ell,m}' e^{C_{\ell,m}' d_{\tilde g}(x_0,y_0)} \qquad \text{ for } \ell,m = 0,1,2,\ldots
	\]
	where $C_{\ell,m}'$ only depends only on $M$ and finitely many of the constants $C_\ell'$.
	
	Note $\Delta_x + \Delta_y$ is the Laplace-Beltrami operator on the product manifold $M \times M$ endowed with the product metric. Moreover, $e_p(x)e_q(y)$ for $p,q = 0,1,2,\ldots$ form a Hilbert basis of eigenfunctions of $\Delta_x + \Delta_y$ with
	\[
	-(\Delta_x + \Delta_y)e_p(x) e_q(y) = (\lambda_p^2 + \lambda_q^2) e_p(x) e_q(y).
	\]
	Hence if we write
	\[
	\hat F(p,q) = \iint_{M\times M} F(x,y) \overline{e_p(x) e_q(y)} \, dx \, dy,
	\]
	we have by Sobolev embedding
	\begin{align*}
	&\|\Delta_x^\ell \Delta_y^m F \|_{L^\infty(M\times M)}^2\\
	&\lesssim \| (1 - \Delta_x - \Delta_y)^3 \Delta_x^\ell \Delta_y^m F \|_{L^2(M\times M)}^2 \\
	&= \sum_{p,q} (1 + \lambda_p^2  + \lambda_q^2)^{6} \lambda_p^{4\ell} \lambda_q^{4m} |\hat F(p,q)|^2 \\
	&\leq C_{\ell,m} \sum_{p,q} (1 + \lambda_p^{4(3 + \ell + m + 1)} + \lambda_q^{4(3 + \ell + m + 1)}) |\hat F(\ell,m)|^2\\
	&\leq C_{\ell,m} \left( \| F \|_{L^2(M\times M)}^2 + \| \Delta_x^{3+\ell+m+1} F \|_{L^2(M\times M)}^2 + \| \Delta_y^{3+\ell+m+1} F \|_{L^2(M\times M)}^2 \right).
	\end{align*}
	The desired bounds follow from H\"older's inequality and \eqref{mixed bounds prop 3}.
\end{proof}

		\bibliography{EF}{}

\def\cprime{$'$} \def\cprime{$'$}
\begin{thebibliography}{Wym17b}

\bibitem[CG17]{canzani}
Y.~Canzani and J.~Galkowski.
\newblock On the growth of eigenfunction averages: microlocalization and
  geometry.
\newblock Preprint, 2017.

\bibitem[CS15]{CSPer}
X.~Chen and C.~D. Sogge.
\newblock On integrals of eigenfunctions over geodesics.
\newblock {\em Proc. Amer. Math. Soc.}, 143(1):151--161, 2015.

\bibitem[dC92]{doCarmo}
M.~P. do~Carmo.
\newblock {\em Riemannian geometry}.
\newblock Birkh\"auser Boston, Inc., Boston, MA, 1992.

\bibitem[Goo83]{Good}
A.~Good.
\newblock {\em Local analysis of {S}elberg's trace formula}, volume 1040 of
  {\em Lecture Notes in Mathematics}.
\newblock Springer-Verlag, Berlin, 1983.

\bibitem[Hej82]{Hej}
D.~A. Hejhal.
\newblock Sur certaines s\'eries de {D}irichlet associ\'ees aux g\'eod\'esiques
  ferm\'ees d'une surface de {R}iemann compacte.
\newblock {\em C. R. Acad. Sci. Paris S\'er. I Math.}, 294(8):273--276, 1982.

\bibitem[H{\"o}r90]{HI}
L.~H{\"o}rmander.
\newblock {\em The analysis of linear partial differential operators. {I}}.
\newblock Springer-Verlag, Berlin, second edition, 1990.

\bibitem[Pit08]{Pitt}
N.~J.~E. Pitt.
\newblock A sum formula for a pair of closed geodesics on a hyperbolic surface.
\newblock {\em Duke Math. J.}, 143(3):407--435, 2008.

\bibitem[Rez15]{Rez}
A.~Reznikov.
\newblock A uniform bound for geodesic periods of eigenfunctions on hyperbolic
  surfaces.
\newblock {\em Forum Math.}, 27(3):1569--1590, 2015.

\bibitem[Sog17]{SFIO}
C.~D. Sogge.
\newblock {\em Fourier integrals in classical analysis}, volume 210 of {\em
  Cambridge Tracts in Mathematics}.
\newblock Cambridge University Press, Cambridge, 2nd edition, 2017.

\bibitem[SXZ17]{Gauss}
C.~D. Sogge, Y.~Xi, and C.~Zhang.
\newblock Geodesic period integrals of eigenfunctions on {R}iemannian surfaces
  and the {G}auss-{B}onnet theorem.
\newblock {\em Camb. J. Math.}, 5(1):123--151, 2017.

\bibitem[TZ13]{TZ}
J.~Toth and S.~Zelditch.
\newblock Quantum ergodic restriction theorems: Manifolds without boundary.
\newblock {\em Geometric and Functional Analysis}, 23(2):715--775, 2013.

\bibitem[Wym17a]{emmett2}
E.~Wyman.
\newblock Explicit bounds on integrals of eigenfunctions over curves in
  surfaces of nonpositive curvature.
\newblock preprint, 2017.

\bibitem[Wym17b]{emmett1}
E.~Wyman.
\newblock Integrals of eigenfunctions over curves in surfaces of nonpositive
  curvature.
\newblock preprint, 2017.

\bibitem[Xi17a]{GP}
Y.~Xi.
\newblock Improved generalized periods estimates on {R}iemannian surfaces with
  nonpositive curvature.
\newblock Preprint, 2017.

\bibitem[Xi17b]{inner}
Y.~Xi.
\newblock Inner product of eigenfunctions over curves and generalized periods
  for compact riemannian surfaces.
\newblock Preprint, 2017.

\bibitem[Zel88]{Zel}
S.~Zelditch.
\newblock Selberg trace formulae, pseudodifferential operators, and geodesic
  periods of automorphic forms.
\newblock {\em Duke Math. J.}, 56(2):295--344, 1988.

\bibitem[Zel92]{ZelK}
S.~Zelditch.
\newblock Kuznecov sum formulae and {S}zeg{\H o}\ limit formulae on manifolds.
\newblock {\em Comm. Partial Differential Equations}, 17(1-2):221--260, 1992.

\end{thebibliography}
		\bibliographystyle{alpha}
	\end{document}